\documentclass[a4paper,10pt]{article}

\usepackage[utf8]{inputenc}
\usepackage{amsmath,amssymb,amsfonts,amsthm}
\usepackage{color}
\usepackage{textcomp}
\usepackage{stmaryrd}
\usepackage{bbm}
\usepackage{dsfont}
\usepackage[titletoc]{appendix}

\newcommand{\E}{\mathbb{E}}

\newcommand{\N}{\mathbb{N}}
\renewcommand{\P}{\mathbb{P}}
\newcommand{\Q}{\mathbb{Q}}
\newcommand{\R}{\mathbb{R}}

\newcommand{\ZZ}{\mathbb{Z}}

\newcommand{\Ac}{\mathcal{A}}
\newcommand{\Bc}{\mathcal{B}}

\newcommand{\Fc}{\mathcal{F}}

\newcommand{\Kc}{\mathcal{K}}
\newcommand{\Lc}{\mathcal{L}}
\newcommand{\Mc}{\mathcal{M}}

\newcommand{\Md}{\mathrm{M}}
\newcommand{\ie}{\textit{i.e.}}
\newcommand{\eg}{\textit{e.g.}}

\def \Om{\Omega}

\def \om{\omega}
\def \omb{\bar{\omega}}

\def \Md{\mathrm{M}}

\newcommand{\x}{\times}
\newcommand{\ox}{\otimes}

\newcommand{\as}{\mathrm{a.s.}}

\newcommand{\X}[1]{X^{#1}}

\newcommand{\e}[1]{\mathrm{e}^{#1}}

\newcommand{\bmax}{\bar{b}}
\newcommand{\dmax}{\bar{d}}
\newcommand{\dmin}{\underline{d}}
\newcommand{\TV}[1]{\lVert{#1}\rVert_{TV}}

\newtheorem{theorem}{Theorem}
\newtheorem{lemma}[theorem]{Lemma}
\newtheorem{proposition}[theorem]{Proposition}
\newtheorem{corollary}[theorem]{Corollary}

\newtheorem{hypothesis}{Hypothesis}

\theoremstyle{remark} \newtheorem*{remark}{Remark}

\newcommand{\rmi}{{\rm (i)$\>\>$}}
\newcommand{\rmii}{{\rm (ii)$\>\>$}}
\newcommand{\rmiii}{{\rm (iii)$\>\>$}}
\newcommand{\rmiv}{{\rm (iv)$\>\>$}}

\def\restriction#1#2{\mathchoice
              {\setbox1\hbox{${\displaystyle #1}_{\scriptstyle #2}$}
              \restrictionaux{#1}{#2}}
              {\setbox1\hbox{${\textstyle #1}_{\scriptstyle #2}$}
              \restrictionaux{#1}{#2}}
              {\setbox1\hbox{${\scriptstyle #1}_{\scriptscriptstyle #2}$}
              \restrictionaux{#1}{#2}}
              {\setbox1\hbox{${\scriptscriptstyle #1}_{\scriptscriptstyle #2}$}
              \restrictionaux{#1}{#2}}}
\def\restrictionaux#1#2{{#1\,\smash{\vrule height .8\ht1 depth .85\dp1}}_{\,#2}}

\title{On the link between infinite horizon control and quasi-stationary distributions}
\author{Nicolas Champagnat\thanks{Universit\'e de Lorraine, Institut Elie Cartan de Lorraine,
    UMR 7502, Vand\oe uvre-l\`es-Nancy, F-54506, France;
    Nicolas.Champagnat@inria.fr}~\thanks{CNRS, Institut Elie Cartan de Lorraine, UMR 7502,
    Vand\oe uvre-l\`es-Nancy, F-54506, France}~\thanks{Inria, TOSCA, Villers-l\`es-Nancy,
    F-54600, France}~, Julien Claisse\thanks{Ecole Polytechnique, CMAP, Palaiseau, F-91128, France; claisse@cmap.polytechnique.fr}}

\begin{document}

\maketitle

\begin{abstract}
    We study infinite horizon control of continuous-time non-linear branching processes with almost sure extinction for
    general (positive or negative) discount. Our main goal is to study the link between infinite horizon control of these processes and
    an optimization problem involving their quasi-stationary distributions and the corresponding extinction rates. More precisely, we
    obtain an equivalent of the value function when the discount parameter is close to the threshold where the value function becomes
    infinite, and we characterize the optimal Markov control in this limit. To achieve this, we present a new proof of the dynamic
    programming principle based upon a pseudo-Markov property for controlled jump processes. We also prove the convergence to a
    unique quasi-stationary distribution of non-linear branching processes controlled by a Markov control conditioned on
    non-extinction.
\end{abstract}

 \vspace{1mm}


 \vspace{2mm}


\section{Introduction}
\label{sec:introduction}

Given a family of controlled stochastic processes $(X^{x,\alpha}_t,t\geq 0)$ in some measurable state space $(S,\mathcal{S})$, adapted
to a filtration $(\mathcal{F}_t)_{t\geq 0}$, where $x\in S$ is the initial value of $X^{x,\alpha}$, $\alpha=(\alpha_s,s\geq 0)$ is a
$(\mathcal{F}_t)_{t\geq 0}$-predictable process belonging to a given set $\mathcal{A}$ of admissible controls, the classical infinite horizon
control problem is formulated as follows: given $\beta<0$ and a bounded measurable function $f$, one defines the value function as
\begin{equation}
  \label{eq:general-ergodic-control}
  v_\beta(x)=\sup_{\alpha\in\mathcal{A}}\E\left[\int_0^\infty e^{\beta t}
    f(X^{x,\alpha}_t,\alpha_t)\,dt\right].
\end{equation}
The infinite horizon control problem consists in characterizing the optimal control $\alpha^*$ realizing the above supremum (when it exists).

In the case where $(X^{x,\alpha}_t,t\geq 0)$ is an ergodic Markov process for all Markov control $\alpha$, i.e.\ all control of the
form $\alpha_s=a(X^{x,\alpha}_{s-})$ for some measurable function $a$, the infinite horizon control problem when $\beta\uparrow 0$ can be
formulated as an optimization problem on (ergodic) invariant distributions. More precisely, If we call $\mathcal{A}_M$ the set of
Markov controls, using the ergodic theorem,
one expects that, when $\beta\uparrow 0$,
\begin{equation}
  \label{eq:heuristics}
  v_\beta(x)\sim \frac{1}{\beta}\sup_{\alpha\in\mathcal{A}_M}\mu^\alpha(f^\alpha),
\end{equation}
where, for all $\alpha\in\mathcal{A}_M$, $f^\alpha:x\mapsto f(x,a(x))$ and $\mu^\alpha$ is the invariant distribution of
$X^{x,\alpha}$ for all $x\in S$. 
 For further details, we refer the reader to, e.g., \cite{puterman-94} for discrete-time Markov
 chains, \cite{guo-hernandez-09} for the continuous-time case and \cite{arapostathis-borkar-ghosh-12} for diffusion processes.

If one assumes that for all $\alpha\in\mathcal{A}$ and all $x\in S$, the process $X^{x,\alpha}$ is a.s.\ absorbed after a finite time
at some point $\partial$, and if one assumes $f(\partial)=0$ and $f(x)>0$ for all $x\in S\setminus\{\partial\}$, then one expects
that the ergodic control problem~\eqref{eq:general-ergodic-control} favors controls $\alpha$ for which the absorption time of
$X^{x,\alpha}$ is longer. The relation~\eqref{eq:heuristics} clearly does not apply in this case since $\mu_\alpha=\delta_\partial$ and
$f(\partial)=0$. This suggests that this new problem needs to be studied using quasi-stationary distributions (QSD for short) instead
of the stationary distribution $\delta_\partial$.

For a Markov process $(Y_t,t\geq 0)$ taking values in $S\cup\{\partial\}$ a.s.\ absorbed in finite time at $\partial$, a
probability measure $\pi$ on $S$ is said to be quasi-stationary if
$$
\P_\pi(Y_t\in \cdot\mid t<\tau)=\pi,\qquad\forall\,t\geq 0,
$$
where $\tau:=\inf\{t\geq 0:Y_t=\partial\}$ is the absorption time of $Y$ and $\P_\pi$ is the law of $Y$ given that $Y_0\sim\pi$. 
It is well-known ~\cite{meleard-villemonais-12} that, if $\pi$ is a QSD, then the absorption time starting from $\pi$ is
exponentially distributed with some parameter $\lambda>0$ called the absorption rate of the QSD $\pi$:
$$
\P_\pi(t<\tau)=e^{-\lambda t},\quad\forall\,t\geq 0.
$$
Recently, new criteria to prove the existence, uniqueness and exponential convergence of conditional distribution
of general absorbed Markov processes were obtained in~\cite{champagnat-villemonais-15}.

The goal of this article is to make the link between infinite horizon control of absorbed processes and quasi-stationary
distributions, by proving an extension of~\eqref{eq:heuristics} in the case where $f(\partial)=0$.

We will restrict in this work to continuous-time controlled processes in $\mathbb{Z}_+:=\{0,1,2,\ldots\}$ which are \emph{non-linear
  branching processes}. We call them \emph{branching processes} since they allow simultaneous births from a single individual, as in
the continuous-time counterpart of Galton-Watson processes, and deaths only occur singly. In addition, the process is absorbed at the
state $\partial=0$, when the population goes extinct. However, we consider more general processes than processes satisfying the
branching property (processes with linear birth and death rates), which explains the term \emph{non-linear}. In such a general
settings, very few results on QSD are known, except for processes satisfying the branching property~\cite{athreya-ney-72},
 and for state-dependent branching processes in discrete time~\cite{gosselin-01}, but the last reference does not
  obtain uniform convergence of conditional distributions, which is needed for our analysis. A particular case is given by classical
birth and death processes, for which quasi-stationary properties have been studied for a long
time~\cite{cavender-78,vandoorn-91,ferrari-martinez-al-92}. More recently, the topic of QSD for birth and death processes has been
further studied in~\cite{Martinez-Martin-Villemonais2012,champagnat-villemonais-15,champagnat-villemonais-PNM-15}
and~\cite{chazottes-al-15}. The first three of the last references study the exponential convergence in total variation of
conditional distributions to a unique QSD, uniformly w.r.t.\ the initial distribution. The new results we obtain here on QSD in the
setting of non-linear branching processes are exactly of this type.

Non-linear branching processes can be controlled in the sense that the birth and death rates and the progeny distribution may depend
on a control parameter. Note that a controlled process may not satisfy the branching property even if each process with a constant
control satisfies it. Hence, for the sake of simplicity, we will call our processes \emph{controlled branching processes} and omit
the \emph{non-linear} which is implicit.

The optimal control of stochastic or deterministic population dynamics like non-linear branching processes 
  is an important topic in various biological domains. We can cite for example conservation biology, where the controller tries to favor
the survival of an endangered species~\cite{houston-namara-99,houston-namara-01} 
, or medecine and agronomy, where the controller wants to favor the
extinction of a population of pathogens in some disease or of pests in an agricultural process~\cite{possingham-shea-00}.
Other examples of control problems in fishery, agronomy, bio-reactors or tumor growth may be found e.g.\
in~\cite{lenhart-workman-07,anita-al-11}.
  
Section~\ref{sec:model} is devoted to the definition of the controlled non-linear branching processes and to the proof of preliminary
properties. Using the criteria of~\cite{champagnat-villemonais-15}, we also state in Section~\ref{sec:model} and prove in
Section~\ref{sec:QSD} that for all Markov control $\alpha$, the controlled branching process $X^{x,\alpha}$ admits a unique
quasi-stationary distribution $\pi^\alpha$ with absorption rate $\lambda^\alpha>0$, and that the conditional distributions converge
exponentially and uniformly in total variation to the QSD. We extend in Section~\ref{sec:main-result} the problem of infinite horizon
control~\eqref{eq:general-ergodic-control} to positive values of $\beta$, and we also state our main results on
  infinite horizon control, among which the fact that, if $f\geq 0$ and
$f(0,\cdot)\equiv 0$, then, when $\beta\rightarrow\lambda_*:=\inf_{\alpha\in\mathcal{A}_M}\lambda^\alpha$,
$$
v_\beta(x)\sim\frac{1}{\lambda_*-\beta}\sup_{\alpha\in\mathcal{A}_M\text{ s.t.\ }\lambda^\alpha=\lambda_*}\pi^\alpha(f^{\alpha})\eta^{\alpha}(x),
$$
where $\eta^{\alpha}$ is the eigenvector of the generator of $\X{\alpha}$ corresponding to the eigenvalue $\lambda^{\alpha}$.
This result deals with control problems favoring survival of the population. We also state similar results on
  control problems in favor of extinction. In Section~\ref{sec:ergodic-control}, we prove that the value function is solution to the
Hamilton-Jacobi-Bellman equation, and that the optimal control belongs to the class of Markov controls. These results are to a large
extent classical, but the proofs we give are original. Our analysis relies on the analogy with the theory of controlled diffusions
using a Poisson random measure playing the role of the Brownian motion. Section~\ref{sec:QSD} is devoted to proofs of our results on
QSD. Finally, we give the proofs of the main results of Section~\ref{sec:main-result} in Section~\ref{sec:proof}.

\section{Controlled continuous-time non-linear branching processes}
  \label{sec:model}

\subsection{Definition and first properties}

Let $A$ be the control space that is assumed to be Polish. 
Consider a population which evolves like a controlled non-linear branching process. More precisely, if $n\in\ZZ_+:=\{0,1,2,\ldots\}$
is the size of the population, then the (global) death rate is given by $d_n:A\to\R_+$, the (global) birth rate by $b_n:A\to\R_+$ and
the probability to have $k\in\N:=\{1,2,\ldots\}$ offsprings by $p_{n,k}:A\to[0,1]$, where
\begin{equation*}
  \sum_{k=1}^{+\infty} {p_{n,k} (a)} = 1,\quad \forall\,n\in\ZZ_+,\, a\in A.
\end{equation*}
Moreover, we suppose that
\begin{gather*}
  b_{0}(a)=d_{0}(a)=0, \quad \forall\, a\in A,
\end{gather*}
so that $\partial=0$ is an absorbing state of the population process.

Let $Q_1$ and $Q_2$ be two independent Poisson random measures on $\R_+\x\R_+$ with Lebesgue intensity measures. Denote by $(\Fc_s)_{s\geq 0}$ the filtration generated by $Q_1$ and $Q_2$, that is,
\begin{equation*}
 \Fc_s := \sigma\left(Q_1(B),Q_2(B);\ B\in \Bc([0,s])\ox\Bc(\R_+)\right).
\end{equation*}
Let $\Ac$ be the collection of $(\mathcal{F}_s)_{s\geq 0}$-predictable processes valued in $A$. An element of $\Ac$ is called an
admissible control.

We want to characterize the population $\X{x,\alpha}$ controlled by $\alpha\in\Ac$
starting from $x\in\ZZ_+$ as the solution of the following SDE:
\begin{multline}\label{eq:eds}
  \X{x,\alpha}_s = x + \int_{(0,s]\x\R_+}{\left( \sum_{k\geq 1} { k \mathds{1}_{I_k\left(\X{x,\alpha}_{\theta-},\alpha_{\theta}\right)}(z) } \right) Q_1(d\theta,d z)} \\
  - \int_{(0,s]\x\R_+} {\mathds{1}_{\left[0,d\left(\X{x,\alpha}_{\theta-},\alpha_{\theta}\right)\right)}(z) \,Q_2(d\theta,d z)},
  \quad \forall\,s\geq 0, \P-\as,
\end{multline}
where $\X{x,\alpha}_{\theta-}$ denotes the left-hand limit of $X^{x,\alpha}$ at time $\theta$ and
\begin{equation*}
  I_{n,k}(a) := \left[b_n(a) \sum_{l=1}^{k-1} {p_{n,l}(a)}\, ,\ b_n(a) \sum_{l=1}^{k} {p_{n,l}(a)}\right).
\end{equation*}
The proposition below ensures that, under suitable conditions, the process $\X{x,\alpha}$ is well-defined.

\begin{hypothesis}\label{hyp:population} Assume that:\\
  \rmi $b_n$, $d_n$ and $p_{n,k}$ are measurable; \\
  \rmii there exists $\bmax>0$ such that
  \begin{equation*}
    b_n(a)\leq \bmax n,\quad \forall\,n\in\ZZ_+,\, a\in A;
  \end{equation*}
  \rmiii there exists $(\dmax_n)_{n\in\ZZ_+}\in(\R_+)^{\ZZ_+}$ such that
  \begin{equation*}  
    d_n(a)\leq \dmax_n,\quad \forall\,n\in\ZZ_+,\, a\in A;
  \end{equation*}
  \rmiv there exists $M>0$ such that 
  \begin{equation*}
    \sum_{k=1}^{\infty} {k p_{n,k}(a)} \leq M,\quad \forall\,n\in\ZZ_+,\, a\in A.
  \end{equation*}
\end{hypothesis}

Assumption~(ii) means that the birth rate per individual is bounded by $\bmax$, regardless of the population size and the control.
Assumption~(iv) means that the number of offspring at each birth event has bounded first moment, regardless of the population size and
the control.

\begin{proposition}\label{prop:welldefined}
 Under Hypothesis~\ref{hyp:population}, there exists a unique (up to indistinguishability) c\`adl\`ag and adapted process solution of~\eqref{eq:eds}. In addition, it satisfies
 \begin{equation}\label{eq:moment}
  \E\left[ \sup_{0\leq\theta\leq s} {\left\{\X{x,\alpha}_\theta\right\}} \right] \leq x \e{\bmax M s}.
 \end{equation}
\end{proposition}

\begin{proof}
 Since the process $X^{x,\alpha}$ is piecewise constant, it can be constructed for each $\omega$ by 
 considering
 recursively the sequence of atoms of the Poisson measures for which the indicator functions in the integrals are non-zero. This
 proves the existence and uniqueness of the solution of~\eqref{eq:eds}, until the first accumulation point of the sequence of jump
 times. Therefore, in order to ensure existence and uniqueness of a global solution of~\eqref{eq:eds}, one needs to prove that this
 first accumulation point is infinite. 
 Since Hypothesis~\ref{hyp:population} ensures that the jump rates are bounded as long as the size of the population is, 
 it suffices to prove that $\tau_n\rightarrow+\infty$ as $n\rightarrow+\infty$, where
 \begin{equation*}
  \tau_n := \inf{\big\{s\geq 0;\ \X{x,\alpha}_s\geq n\big\}}.
 \end{equation*}
 We shall prove this and~\eqref{eq:moment} simultaneously. Denote $\X{*}_s := \sup_{0\leq\theta\leq s}{\{\X{x,\alpha}_{\theta}\}}$.
 One clearly has
 \begin{equation*}
   \X{*}_{s\wedge\tau_n} \leq x + \int_{(0,s\wedge\tau_n]\x\R_+}  {  \sum_{k\geq 1} { k
       \mathds{1}_{I_k\left(\X{x,\alpha}_{\theta-},\alpha_{\theta}\right)}\left(z\right) }\ Q_1\left(d \theta, d z\right)  }.
 \end{equation*}
 It yields
 \begin{align*}
   \E\left[\X{*}_{s\wedge\tau_n}\right]
   & \leq x + \E\left[\int_0^{s\wedge\tau_n}  { b\left(\X{x,\alpha}_{\theta},\alpha_{\theta}\right) \sum_{k\geq 1} { k
         p_{k}\left(\X{x,\alpha}_{\theta},\alpha_{\theta}\right) }\ d\theta } \right] \\
   & \leq x + \bmax M \E\left[\int_0^{s} { \X{*}_{\theta\wedge\tau_n}\ d\theta } \right].
 \end{align*}
 Using Gronwall's lemma, we obtain
 \begin{equation*}
  \E\left[ \X{*}_{s\wedge\tau_n}\right]\leq x\e{\bmax M s}.
 \end{equation*}
 Since the r.h.s. does not depend on $n$, we deduce that $\tau_n$ converges almost surely to infinity. By Fatou's lemma, we conclude
 that the inequality~\eqref{eq:moment} holds.
\end{proof}

The last result shows the existence and uniqueness of the solution of~\eqref{eq:eds} except on the event (of zero probability) where
there is accumulation of jump events in a bounded time interval. We shall call a time of accumulation of jumps an explosion time. For
questions of measurability, we need to define precisely the process $X^{x,\alpha}$ on this event. To this aim, we assume that this
process takes values in $\ZZ_+\cup\{\infty\}$, and we pose $X_s=\infty$ after the first explosion time. The process defined this way
is clearly adapted w.r.t.\ $(\mathcal{F}_s)_{s\geq 0}$. A more standard definition of $X^{x,\alpha}$, for example constant equal to
$0$ on the event of explosion, would require to complete the filtration to obtain an adapted process, but this would pose problems to
prove the pseudo-Markov property and the dynamic programming principle of Section~\ref{sec:ergodic-control}.

\subsection{Markov controls and quasi-stationary distributions}
In the sequel, we need to consider a subclass of the admissible controls, the so-called Markov controls. For any sequence $(\alpha_n)_{n\in\ZZ_+}$ valued in $A$, consider the following SDE: 
\begin{multline}\label{eq:edsm}
  X_s = x + \int_{(0,s]\x\R_+}{\left( \sum_{k\geq 1} { k \mathds{1}_{I^{\alpha}_k\left(X_{\theta-}\right)}(z) } \right) Q_1(d\theta,d z)} \\
  - \int_{(0,s]\x\R_+} {\mathds{1}_{\left[0,d^{\alpha}\left(X_{\theta-}\right)\right)}(z) \,Q_2(d\theta,d z)},
  \quad \forall\,s\geq 0, \P-\as,
\end{multline}
where $d^\alpha(n):=d_n(\alpha_n)$ and $I^\alpha_k(n):=I_{n,k}(\alpha_n)$ for all $k,n\geq 0$. By the arguments of
Proposition~\ref{prop:welldefined}, there is a unique (up to indistinguishability) c\`adl\`ag and adapted process solution to the SDE
above. In addition, this is the controlled process, solution to~\eqref{eq:eds} for the (admissible) control
$(s,\om)\mapsto\alpha(X_{s-}(\om))$, where we use the notation $\alpha(n)=\alpha_n$. As a consequence, by abuse of notation, $\alpha$
will refer in the sequel both to the sequence $(\alpha_n)_{n\in\ZZ_+}$ and to the corresponding admissible control. Since in this case
$\X{x,\alpha}$ is a Markov process, we call $\alpha$ a Markov control. Denote by $\Ac_M$ the collection of Markov controls.

 For all $\alpha\in\mathcal{A}_M$, the process $X^{x,\alpha}$ is Markov on $\ZZ_+$, with transition rates
 $$
 \begin{cases}
    b^\alpha(n) p^\alpha_{k}(n) & \text{from $n$ to $n+k$ for all $n\geq 1$ and $k\geq 1$,} \\
    d^\alpha(n) & \text{from $n$ to $n-1$ for all $n\geq 1$,}\\
    0 & \text{otherwise,}
 \end{cases}
 $$
 where $b^\alpha(n):=b_n(\alpha_n)$ and $p^\alpha_k(n):=p_{n,k}(\alpha_n)$. In other words, its infinitesimal generator is given by
 \begin{equation}\label{eq:operateur}
  \Lc^\alpha u_n := b^{\alpha}(n) \sum_{k=1}^{\infty} {\left(u_{n+k}-u_{n}\right)p^{\alpha}_{k}(n)} + d^{\alpha}(n)\left(u_{n-1} - u_n\right),
 \end{equation}
 for all bounded real-valued sequence $(u_n)_{n\in\ZZ_+}$.
  
 We give now a new set of assumptions that, together with Hypothesis~\ref{hyp:population}, imply the existence and uniqueness of a QSD
 for $X^{x,\alpha}$ for any Markov control $\alpha\in\mathcal{A}_M$.
 
 \begin{hypothesis}\label{hyp:qsd}
  Assume that: \\
  \rmi there exists $\epsilon>0$ and $\dmin>0$ such that
  \begin{equation*}
    d_n(a)\geq \dmin n^{1+\epsilon},\quad \forall\,n\in\ZZ_+,\, a\in A;
  \end{equation*}
  \rmii for all $y,z\in\N$,
  \begin{equation*}
    \inf_{\alpha\in\Ac_M} {\P_y\left(\X{\alpha}_1=z\right)}>0.
  \end{equation*}
 \end{hypothesis}
 
 Point~(i) implies that the absorption state 0 is accessible from any other state and any choice of the control. Further, together
 with the point (ii) of Hypothesis~\ref{hyp:population}, it ensures that the population goes extinct almost surely. Point~(ii) is an
 irreducibility property of the branching process, away from the absorbing point, uniformly w.r.t.\ the control. In particular, if we
 assume that for all $n\in\N$, there exists $k\in\N$ such that
 \begin{equation*}
   \inf_{a\in A} {p_{n,k}(a) b_n(a)}>0,
 \end{equation*}  
 this condition is satisfied.
 
 \begin{theorem}
   \label{thm:QSD-short-version}
   Under Hypotheses~\ref{hyp:population} and \ref{hyp:qsd}, for all $\alpha\in\mathcal{A}_M$, there exists a unique QSD $\pi^\alpha$
   on $\N$ for the process $X^\alpha$, and there exist constants $C,\gamma>0$ such that, for
   all $\alpha\in\mathcal{A}_M$ and all $x\in\N$,
   \begin{equation*}
     \left\|\P(X^{x,\alpha}\in\cdot\mid t<\tau^{x,\alpha})-\pi^\alpha\right\|_{TV}\leq Ce^{-\gamma t},\quad\forall t\geq 0,
   \end{equation*}
   where $\|\cdot\|_{TV}$ is the total variation norm and $\tau^{x,\alpha}$ is the extinction time of $\X{x,\alpha}$, that is,
   \begin{align*}
     \tau^{x,\alpha} := \inf\left\{s\geq 0:\ \X{x,\alpha}_s = 0\right\}.
   \end{align*}
 \end{theorem}

 This theorem will be proved with other properties related to QSD, gathered in Proposition~\ref{thm:QSD} below, in
 Section~\ref{sec:QSD}. For all $\alpha\in\mathcal{A}_M$, we denote by $\lambda^\alpha>0$ the absorption rate of
   $\pi^\alpha$.
 
 \begin{proposition}
  \label{thm:QSD}
  Under Hypotheses~\ref{hyp:population} and~\ref{hyp:qsd},
  for all $\alpha\in\mathcal{A}_M$, there exists a map $\eta^\alpha$ from $\N$ to $(0,+\infty)$ such that, for all $x\in\N$,
  \begin{equation}
    \label{eq:etadef}
    \sup_{x\in\N}\left|\P(t<\tau^{x,\alpha})e^{\lambda^\alpha t}-\eta^\alpha(x)\right|\leq C' e^{-\gamma t},
  \end{equation}
  where
  the constant $C'$ does not depend on $\alpha\in\mathcal{A}_M$. In addition, $\pi^\alpha(\eta^\alpha)=1$ and
  $\mathcal{L}^\alpha\eta^\alpha=-\lambda^\alpha\eta^\alpha$, where $\mathcal{L}^\alpha$ is the generator of $X^\alpha$.
\end{proposition}

\section{Main results}
\label{sec:main-result}

\subsection{Optimality criterion}
\label{sec:criteria}

Let $\beta\in\R$ be the discount factor and $f:\ZZ_+\x A\to\R$ be the running cost.

\begin{hypothesis}\label{hyp:cost}Assume that:\\
  \rmi $f$ is measurable and non-negative;\\
  \rmii $f$ is bounded and satisfies
  \begin{equation*}
    f(0,a) = 0,\quad \forall\,a \in A.
  \end{equation*}  
\end{hypothesis}
 
Consider the $\beta$-discounted infinite horizon criterion given by the cost/reward function, defined for all $\alpha\in\mathcal{A}$ and
$x\in\ZZ_+$, by
\begin{equation*}
  J_{\beta}(x,\alpha)
  := \E\left[\int_0^{+\infty} {\e{\beta s} f\left(\X{x,\alpha}_s,\alpha_s\right) ds}\right] 
  = \E_{x}\left[\int_0^{+\infty} {\e{\beta s} f\left(\X{\alpha}_s,\alpha_s\right) ds}\right],
\end{equation*}
where we denote by $\P_x$ the law of $X^{x,\alpha}$ and $\E_x$ the corresponding expectation, and use the generic notation $X^\alpha$
for the process $X^{x,\alpha}$ under $\P_x$. Notice that the cost/reward function can also be written as
\begin{equation*}
  J_{\beta}(x,\alpha) 
  = \E\left[\int_0^{\tau^{x,\alpha}} {\e{\beta t} f\left(\X{x,\alpha}_s,\alpha_s\right) ds}\right]
  = \E_{x}\left[\int_0^{\tau^{\alpha}} {\e{\beta t} f\left(\X{\alpha}_s,\alpha_s\right) ds}\right].
\end{equation*}

Under Hypothesis~\ref{hyp:cost}, the cost/reward function is well-defined and takes values in $[0,+\infty]$. The next proposition provides estimates for this function. 

\begin{proposition}\label{prop:controlqsd}
  Under Hypotheses~\ref{hyp:population}, \ref{hyp:qsd} and~\ref{hyp:cost}, there exists a constant $C>0$ such that, for all
  $\alpha\in\mathcal{A}_M$
  and $\beta<\lambda^{\alpha}$,
  \begin{equation}
    \label{eq:controlqsd}
    J_\beta(x,\alpha)\leq C\,\frac{\|f\|_\infty}{\lambda^\alpha-\beta},\quad\forall x\in\N.
  \end{equation}
\end{proposition}

\begin{proof}
 The Fubini-Tonelli theorem implies that
 \begin{align*}
  J_\beta\left(x,\alpha\right)  
  = \int_0^{+\infty} {\e{\beta t} \E_x\left[ f^{\alpha}\left(\X{\alpha}_t\right)\right] dt} 
  & \leq \lVert f\rVert_{\infty} \int_0^{+\infty} {\e{\beta t} \P_x\left(t<\tau^{\alpha}\right) dt} \\
  & \leq  (C'+\eta^\alpha(x)) \frac{\lVert f\rVert_{\infty}}{\lambda^\alpha-\beta},
 \end{align*}
where the last inequality follows from Proposition~\ref{thm:QSD}. Since this proposition also implies that
$\eta^\alpha(x)\leq C'+1$ for all $x\in\N$ and $\alpha\in\Ac_M$, the inequality~\eqref{eq:controlqsd} is proved 
with $C:= 2C'+1$. 
\end{proof}

\subsection{Extinction-oriented problems}
\label{sec:extinction}

In this section, we study stochastic control problems that favor extinction of the population. Since
$f(0,\cdot)\equiv 0$, we aim to minimize the optimality criterion, also called cost function, given in Section~\ref{sec:criteria}.

Let us define the infinite horizon value function by
\begin{equation*}
 v_{\beta}\left(x\right) := \inf_{\alpha\in\Ac} {J_{\beta}\left(x,\alpha\right)}.
\end{equation*}
 Before giving the results of this section, we need to make some assumptions.
 
\begin{hypothesis}\label{hyp:main}
 Assume that: \\
 \rmi $A$ is  a compact metric space;\\
 \rmii $b_n$, $d_n$, $p_{n,k}$ and $f(n,\cdot)$ are continuous on $A$ for all $n\in\N$.
\end{hypothesis}

\begin{hypothesis}\label{hyp:main-inf}
 Assume that there exists $x\in\N$ such that $f(x,a)>0$ for all $a\in A$.
\end{hypothesis} 

Hypothesis~\ref{hyp:main} is a classical assumption in the theory of stochastic control that ensures in particular the existence of
optimal Markov controls. 
  Note that, we could have allowed the control space $A(n)$ to depend on the position $n$
  of the process, by setting $f(n,a)=+\infty$ if $a\in A\setminus A(n)$. In this setting, our main results would hold true if
  Hypothesis~\ref{hyp:main}~(i) is replaced by the hypothesis that $A(n)$ is a compact metric space for all $n\in\ZZ_+$. However, we
  restrict here to the case of a control space independent of the position of the process to keep the presentation and notations
  simple.
Hypothesis~\ref{hyp:main-inf} is not really restrictive. Indeed, if it is not satisfied, then the value function is
identically zero since one can easily construct a Markov control such that the corresponding cost function vanishes.

The next result 
gives the behaviour of the value function as $\beta\uparrow\lambda^*$ where
\begin{equation*}
 \lambda^* := \sup_{\alpha\in\Ac_M} {\lambda^\alpha}.
\end{equation*}

\begin{theorem}\label{thm:main-inf}
  Under Hypotheses~\ref{hyp:population}, \ref{hyp:qsd}, \ref{hyp:cost}, \ref{hyp:main} and \ref{hyp:main-inf}, for all $\beta<\lambda^*$, there exists an
  optimal Markov control for the infinite horizon problem, that is, there exists $\alpha_{\beta}\in\Ac_M$ such that $v_\beta=J_\beta(\cdot,\alpha_\beta)$. 
  In addition, it holds
  \begin{equation*}
    \lim_{\substack{\beta\to\lambda^* \\ \beta<\lambda^*}} {\left(\lambda^* - \beta\right)v_\beta(x)} =
    \inf_{\substack{\alpha\in\Ac_M \\ \lambda^\alpha=\lambda^*}} {\left\{\pi^{\alpha}(f^{\alpha})\eta^\alpha(x)\right\}},
  \end{equation*}
  where for all $\alpha\in\mathcal{A}_M$, the function $f^\alpha$ is defined by $f^\alpha(n)=f(n,\alpha_n)$ for all $n\in\N$.
  Further, there exists $\alpha^*\in\Ac_M$ such that $\lambda^{\alpha^*}=\lambda^*$ and
  \begin{equation*}
    \inf_{\substack{\alpha\in\Ac_M \\ \lambda^\alpha=\lambda^*}} {\left\{\pi^{\alpha}(f^{\alpha})\eta^\alpha(x)\right\}} =
    \pi^{\alpha^*}(f^{\alpha^*})\eta^{\alpha^*}(x),\quad \forall x\in\N.
  \end{equation*}
\end{theorem}

The theorem above ensures that the $\beta$-discounted infinite horizon problem turns into an optimization problem on the QSD of the
controlled Markov processes as $\beta\uparrow\lambda^*$. This new problem consists in minimizing the integral of the running cost
w.r.t.\ the QSD over all the processes with the highest extinction rate. From that perspective, it is analoguous to the classical result
\eqref{eq:heuristics}.
However, the optimisation problem on the QSD is not homogeneous as it depends also on the starting point through the eigenvector
$\eta^{\alpha}$. Yet the corresponding optimal control $\alpha^*$ does not depend on the starting point.

As a by-product, Theorem~\ref{thm:main-inf} shows that the extinction rate cannot be increased above $\lambda^*$ by
  using non-Markov controls.

\begin{corollary}
  Under Hypotheses~\ref{hyp:population}, \ref{hyp:qsd} and \ref{hyp:main}, the population process cannot extinguish faster using a
  non-Markovian control than using a Markov control, in the sense that for all $\alpha\in\Ac$ and $x\in\N$,
 \begin{equation*}
  \int_0^{+\infty} {e^{\lambda^* s} \P_x\left(s<\tau^{\alpha}\right)  ds} = +\infty.
 \end{equation*}
\end{corollary}

\subsection{Survival-oriented problems}

In this section, we provide results similar to Section~\ref{sec:extinction} for a control problem that favors the survival of the population.
Here,
we aim to maximize the optimality criterion of
Section~\ref{sec:criteria}. This result can also be seen as a version of Theorem~\ref{thm:main-inf} for non-positive running costs
(by a simple change of sign).

Let us define the infinite horizon value function by
\begin{equation*}
 w_{\beta}\left(x\right) := \sup_{\alpha\in\Ac} {J_{\beta}\left(x,\alpha\right)}.
\end{equation*}

The next result is anologuous to Theorem~\ref{thm:main-inf} as it  gives the behaviour of the value function as
$\beta\uparrow\lambda_*$ where
\begin{equation*}
 \lambda_* := \inf_{\alpha\in\Ac_M} {\lambda^\alpha}.
\end{equation*}

\begin{theorem}
  \label{thm:main-sup}
  Under Hypotheses~\ref{hyp:population}, \ref{hyp:qsd}, \ref{hyp:cost} and \ref{hyp:main}, for all
  $\beta<\lambda_*$, there exist an optimal Markov control for the infinite horizon problem, and for
  all $x\in\N$,
  \begin{equation*}
    \lim_{\substack{\beta\to\lambda_* \\ \beta<\lambda_*}} \left(\lambda_* - \beta\right) w_\beta(x) 
    = \sup_{\substack{\alpha\in\Ac_M \\ \lambda^\alpha=\lambda_*}} {\left\{\pi^{\alpha}(f^{\alpha})\eta^\alpha(x)\right\}}.
  \end{equation*}
  Further, there exists $\alpha_*\in\Ac_M$ such that $\lambda^{\alpha_*}=\lambda_*$ and
  \begin{equation*}
    \sup_{\substack{\alpha\in\Ac_M \\ \lambda^\alpha=\lambda_*}} {\left\{\pi^{\alpha}(f^{\alpha})\eta^\alpha(x)\right\}} =
    \pi^{\alpha_*}(f^{\alpha_*})\eta^{\alpha_*}(x),\quad\forall x\in\N.
  \end{equation*}
\end{theorem}

The last result also shows that the extinction rate cannot be reduced below $\lambda_*$ by using non-Markov controls.

\begin{corollary}
 Under Hypotheses~\ref{hyp:population}, \ref{hyp:qsd} and \ref{hyp:main}, the population process cannot survive longer using a non-Markovian control, in the sense that for all $\beta<\lambda_*$ and $x\in\N$,
 \begin{equation*}
  \sup_{\alpha\in\Ac} \left\{\int_0^{+\infty} {e^{\beta s} \P_x\left(s<\tau^{\alpha}\right)  ds}\right\} < +\infty.
 \end{equation*}
\end{corollary}

\section{Optimal control problems}
\label{sec:ergodic-control}
  
 The aim of this section is to derive the HJB equations corresponding to the infinite horizon problems and to prove that, under suitable conditions, there exists optimal Markov controls. A key step in our analysis is to establish the corresponding dynamic programming principle for which we provide a new proof based on a pseudo-Markov property (see Lemma~\ref{lem:cond} below). In addition, compared to the classical results available in the literature, we also deal with the case $\beta\geq 0$.

 \subsection{Preliminaries}
 
 Let $\Md$ be the space of integer-valued Borel measures on $\R_+\x\R_+$ that are locally finite, \ie, take finite values on bounded Borel sets. Equipped with the vague topology, it becomes a Polish space (see, \eg,~\cite[App.15.7.7]{kallenberg-86}). Denote by $(\Mc_s)_{s\geq 0}$ its canonical filtration, \ie, $\Mc_s$ is the smallest $\sigma$--algebra such that, for all $B\in\Bc([0,s])\ox\Bc(\R_+)$ bounded,  $\mu\in\Md\mapsto\mu(B)\in\ZZ_+$ is measurable, or equivalently, the $\sigma$--algebra generated by $\mu\in\Md\mapsto\restriction{\mu}{[0,s]}\in\Md$.
 
 Let us define the canonical space as follows: 
 \begin{equation*}
   \Om^\circ = \Md\x\Md,\quad \Fc^\circ_s = \Mc_s\ox\Mc_s,\quad \P^\circ = \Q\ox\Q,
 \end{equation*}
 where $\Q$ is the distribution on $\Md$ of a Poisson random measure on $\R_+\x\R_+$ with Lebesgue intensity measure. We also define
 for all $(\mu_1,\mu_2)\in \Om^\circ$ and $B\in\Bc(\R_+)\ox\Bc(R_+)$,
 \begin{equation*}
   Q^\circ_1\left(\mu_1,\mu_2,B\right) = \mu_1(B),\quad Q^\circ_2\left(\mu_1,\mu_2,B\right) = \mu_2(B).    
 \end{equation*}
  
 Our first goal in this section is to prove that the value function is invariant under change of probability space or Poisson random measures. 
 In particular, we can assume in the sequel that we work in the canonical space, which has the desirable property of being Polish. As a first step,
 the next proposition provides an essential representation of admissible controls.  
 
 \begin{proposition}\label{prop:contcaract}
  A process $\alpha$ is an admissible control if and only if there exists a process $\alpha^{\circ}:\R_+\x \Om^\circ\to A$ predictable w.r.t. $(\Fc^\circ_s)_{s\geq0}$ such that for all $s\geq0$ and $\om\in\Om$,
  \begin{equation}\label{eq:contcaract}
   \alpha_s(\om) = \alpha^{\circ}_s\left(Q_1(\om),Q_2(\om)\right) = \alpha^{\circ}_s\left(\restriction{Q_1}{[0,s)}(\om),\restriction{Q_2}{[0,s)}(\om)\right).
  \end{equation}
 \end{proposition}
 
 \begin{proof}
  It is clear that the process $(s,\om)\mapsto (\restriction{Q_1}{[0,s)}(\om),\restriction{Q_2}{[0,s)}(\om))$ is predictable w.r.t.
  $(\Fc_s)_{s\geq0}$. 
  Thus, if $\alpha$ satisfies~\eqref{eq:contcaract}, it is an admissible control. For the converse implication, we can suppose that $A=\R$ since any Polish space is Borel isomorphic to a Borel subset of $\R$ (see, \eg,~\cite{ikeda-watanabe-89}). We begin with a process $\alpha$ of the form $\alpha=\mathds{1}_{(t_1, t_2]\x F}$ with $F\in\Fc_{t_1}$ and $0<t_1<t_2$. Since $\Fc_{t_1}$ coincides with the $\sigma$--algebra generated by $(\restriction{Q_1}{[0,t_1]},\restriction{Q_2}{[0,t_1]})$, there exists $F^{\circ}\in\Fc^\circ$ such that
  \begin{equation*}
    \alpha_s(\om) = \mathds{1}_{(t_1, t_2]\x F^\circ}\left(s,\restriction{Q_1}{[0,t_1]}(\om),\restriction{Q_2}{[0,t_1]}(\om)\right),\quad \forall\, s\geq0,\, \om\in\Om.
  \end{equation*}
  Hence, the relation~\eqref{eq:contcaract} is satisfied with
  \begin{equation*}
    \alpha^{\circ}_s\left(\mu_1,\mu_2\right) = \mathds{1}_{(t_1, t_2]\x F^{\circ}}\left(s,\restriction{\mu_1}{[0,t_1]},\restriction{\mu_2}{[0,t_1]}\right),\quad \forall\, s\geq0,\, \om\in\Om.
  \end{equation*} 
  Further, $\alpha^{\circ}$ is clearly predictable w.r.t.\ $(\Fc^\circ_s)_{s\geq0}$ as a c\`ag adapted process.
  The corresponding result for $\alpha = \mathds{1}_{\{0\}\x F}$ with $F\in\Fc_0=\{\emptyset,\Om\}$ is obvious. Recall now that the events of the form $(t_1, t_2]\x F$ and $\{0\}\x F$ as above generate the predictable $\sigma$-algebra (see, \eg,~\cite[Thm.IV.64]{dellacherie-meyer-75}). Hence the conclusion follows from a monotone class argument.
 \end{proof}
 
 Let $x\in\ZZ_+$, $\alpha\in\Ac$ and the corresponding $\alpha^{\circ}$ given by Proposition~\ref{prop:contcaract} be fixed. In view of Proposition~\ref{prop:welldefined}, there exists a unique (up to indistinguishability) adapted and c\`adl\`ag process $\X{x,\alpha^\circ}$ in $(\Om^\circ,(\Fc^\circ_s)_{s\geq 0},\P^\circ)$ satisfying
\begin{multline*}
  \X{x,\alpha^\circ}_s = x + \int_{(0,s]\x\R_+}{\left( \sum_{k\geq 1} { k \mathds{1}_{I_k\left(\X{x,\alpha^\circ}_{\theta-},\alpha^\circ_{\theta}\right)}(z) } \right) Q^\circ_1(d\theta,d z)} \\
  - \int_{(0,s]\x\R_+} {\mathds{1}_{\left[0,d\left(\X{x,\alpha^\circ}_{\theta-},\alpha^\circ_{\theta}\right)\right)}(z) \,Q^\circ_2(d\theta,d z)},
  \quad \forall\,s\geq 0,\, \P^{\circ}-\as
\end{multline*}
 Then it is clear that $\X{x,\alpha^\circ}(Q_1,Q_2)$ satisfies the relation~\eqref{eq:eds} in $(\Om,(\Fc_s)_{s\geq0},\P)$. It results from Proposition~\ref{prop:welldefined} that $\X{x,\alpha^\circ}(Q_1,Q_2)$ is $\P$--indistinguishable from $\X{x,\alpha}$. In particular, the distribution of $\X{x,\alpha}$ under $\P$ coincides with the distribution of $\X{x,\alpha^\circ}$ under $\P^{\circ}$. As a consequence, 
 the value function is invariant under change of probability space or Poisson random measures.  
 In the rest of the paper, we assume that we work in the canonical space and we omit the index $\circ$ in the notations.
\medskip

 To conclude this section, we give a technical lemma which plays a crucial role in the proof of the dynamic programming principle
 (DPP for short) below: the so-called pseudo-Markov property for controlled branching processes, as it reduces to the (strong) Markov property in the uncontrolled setting. Let us start with a definition.
 For $\alpha\in\Ac$, we define, for all $t\geq0$ and $\omb\in\Om$, the shifted control $\alpha^{t,\omb}$ as follows:
 \begin{equation*}
  \alpha^{t,\omb}_s\left(\om\right) := \alpha_{t+s}\left(Q_1(\omb)\ox_t Q^t_1(\om), Q_2(\omb)\ox_t Q^t_2(\om)\right), \quad\forall\,s\geq0,\, \om\in\Om,
 \end{equation*}
 where, for $(\mu_1,\mu_2)\in \Md^2$, $\mu_1\ox_t\mu_2$ denotes the element of $\Md$ given by
 \begin{equation*}
  \mu_1\ox_t\mu_2 = \restriction{\mu_1}{[0,t]} + \restriction{\mu_2}{(t,+\infty)},
 \end{equation*}
 and, for $\mu\in\Md$, $\mu^{t}$ denotes the image measure of $\mu$ by the map $(s,z)\in\R_+\x\R_+\mapsto (t+s,z)\in[t,+\infty)\x\R_+$. In particular, $\omb$ being fixed, it is clear that $\alpha^{t,\omb}$ is an admissible control.

 \begin{lemma}\label{lem:cond}
  Under Hypothesis~\ref{hyp:population}, for all $x\in\ZZ_+$, $\alpha\in\Ac$,  $s\in\R_+$, $\tau$ stopping time valued in $[0,s]$ and $\varphi\in(\R_+)^{\ZZ_+}$, it holds
  \begin{equation*}
   \E\left[\varphi\left(\X{x,\alpha}_{s}\right)\Big|\Fc_{\tau}\right] (\omb) =  \E\left[\varphi\left(\X{\X{x,\alpha}_{\tau}(\omb),\alpha^{\tau(\omb),\omb}}_{s-\tau(\omb)}\right)\right],\quad \P\left(d\omb\right)-\as
  \end{equation*} 
 \end{lemma}
 
 \begin{proof}
  By definition of $\X{x,\alpha}$, it is clear that
  \begin{multline*}
    \X{x,\alpha}_s = \X{x,\alpha}_{\tau\wedge s} + \int_{(\tau,\tau\vee s]\x\R_+}{\left( \sum_{k\geq 1} { k \mathds{1}_{I_k\left(\X{x,\alpha}_{\theta-},\alpha_\theta\right)}(z) } \right) Q_1(d\theta,d z)} \\
    - \int_{(\tau,\tau\vee s]\x\R_+} {\mathds{1}_{\left[0,d\left(\X{x,\alpha}_{\theta-},\alpha_\theta\right)\right)}(z) \, Q_2(d\theta,d z)},
    \quad \forall\,s\geq 0, ~\P\mbox{--a.s.}
  \end{multline*}	
  Since $\P$ is a probability measure on the Borel $\sigma$-algebra of a Polish space, there exists $(\P_{\omb})_{\omb \in \Om}$ a family of regular conditional probabilities of $\P$ given $\Fc_{\tau}$ (see, e.g., Karatzas and Shreve~\cite[Sec.5.3.C]{karatzas-shreve-91}). It follows that, for $\P$--a.a. $\omb \in \Om$,
  \begin{equation}\label{eq:condequiv}
    \E^{\P}\left[ \varphi\left(\X{x,\alpha}_s\right)\Big|\,\Fc_\tau\right](\omb) = \E^{\P_{\omb}}\left[ \varphi\left(\X{x,\alpha}_s\right) \right],
  \end{equation}
  and
  \begin{equation*}\label{eq:rcpdprop}
    \P_{\omb} \left( \tau = \tau(\omb), \X{x,\alpha}_{\tau\wedge \cdot} = \X{x,\alpha}_{\tau\wedge \cdot}(\omb), \alpha = \bar{\alpha}^{\tau(\omb),\omb} \right) = 1,
  \end{equation*}
  where 
  \begin{equation*}
   \bar{\alpha}^{\tau(\omb),\omb}_s\left(\om\right) := \alpha_s\left(Q_1(\omb)\ox_{\tau(\omb)} \,Q_1(\om), Q_2(\omb)\ox_{\tau(\omb)} \,Q_2(\om)\right), \quad\forall\,s\geq0,\, \om\in\Om.
  \end{equation*}
  In particular, notice that $\bar{\alpha}^{\tau(\omb),\omb}_{\tau(\omb) + \cdot}(Q_1^{\tau(\omb)},Q_2^{\tau(\omb)})$ coincides with $\alpha^{\tau(\omb),\omb}$ by definition.
  We deduce that, for $\P$--a.a. $\omb\in\Om$,
  \begin{multline*}
    \X{x,\alpha}_s = \X{x,\alpha}_{\tau}(\omb) + \int_{(\tau(\omb), s]\x\R_+}{\left( \sum_{k\geq 1} { k \mathds{1}_{I_k\left(\X{x,\alpha}_{\theta-}, \bar{\alpha}^{\tau(\omb),\omb}_\theta\right)}(z) } \right) Q_1(d\theta,d z)} \\
    - \int_{(\tau(\omb), s]\x\R_+} {\mathds{1}_{\left[0,d\left(\X{x,\alpha}_{\theta-}, \bar{\alpha}^{\tau(\omb),\omb}_\theta\right)\right)}(z) \, Q_2(d\theta,d z)},
    \quad \forall\,s\geq \tau(\omb),\ \P_{\omb}-\as
  \end{multline*}
  In other words, if we denote by $\mu^{-t}$ the image measure of $\restriction{\mu}{[t,\infty)}$ by the map $(s,z)\in[t,+\infty)\x\R_+\mapsto (s-t,z)\in\R_+\x\R_+$, it holds for $\P$--a.a. $\omb\in\Om$,
  \begin{multline*}
    \X{x,\alpha}_{\tau(\omb)+s} = \X{x,\alpha}_{\tau}(\omb) + \int_{(0, s]\x\R_+}{\left( \sum_{k\geq 1} { k \mathds{1}_{I_k\left(\X{x,\alpha}_{\tau(\omb)+\theta-}, \bar{\alpha}^{\tau(\omb),\omb}_{\tau(\omb)+\theta}\right)}(z) } \right) Q_1^{-\tau(\omb)}(d\theta,d z)} \\
    - \int_{(0, s]\x\R_+} {\mathds{1}_{\left[0,d\left(\X{x,\alpha}_{\tau(\omb)+\theta-}, \bar{\alpha}^{\tau(\omb),\omb}_{\tau(\omb)+\theta}\right)\right)}(z) \,Q_2^{-\tau(\omb)}(d\theta,d z)},
    \quad \forall\,s\geq 0,\ \P_{\omb}-\as
  \end{multline*}
  Hence, $\X{x,\alpha}_{\tau(\omb)+\cdot}$ is $\P_{\omb}$--indistinguishable from $\X{\X{x,\alpha}_{\tau}(\omb),\alpha^{\tau(\omb),\omb}}(Q^{-\tau(\omb)}_1, Q^{-\tau(\omb)}_2)$ for $\P$--a.a. $\omb\in\Om$. In particular, the distribution of $\X{x,\alpha}_{\tau(\omb)+\cdot}$ under $\P_{\omb}$ coincides with the distribution of $\X{\X{x,\alpha}_{\tau}(\omb),\alpha^{\tau(\omb),\omb}}$ under $\P$ for $\P$--a.a. $\omb\in\Om$. We then conclude the proof by using~\eqref{eq:condequiv}.
 \end{proof}
 
 \subsection{Minimization problem}

In this section, we derive the HJB equation corresponding to the minization problem and we prove the existence of an optimal Markov control.  All these results are stated under the assumption that the value function $v_\beta$ is bounded, which is implied in our case by Proposition~\ref{prop:controlqsd} if $\beta<\lambda^*$.

\begin{hypothesis}\label{hyp:bounded-inf}
 The value function $v_\beta$ is bounded on $\ZZ_+$.
\end{hypothesis}

We begin with a preliminary result, which plays a crucial role in the proof of the DPP. It is a straightforward application of the
pseudo-Markov property of Lemma~\ref{lem:cond}. 

\begin{proposition}\label{prop:costcond}
 Under Hypotheses~\ref{hyp:population}, \ref{hyp:cost}~(i) and \ref{hyp:bounded-inf}, for all $\alpha\in\Ac$, $x\in\ZZ_+$ and $\tau$ stopping time, it holds
 \begin{equation*}
  J_{\beta}\left(x,\alpha\right) = \E\left[\int_0^{\tau} {\e{\beta s} f\left(\X{x,\alpha}_s,\alpha_s\right) ds} + \e{\beta\tau} J_\beta\left(\X{x,\alpha}_\tau,\alpha^{\tau}\right)\right],
 \end{equation*}
 where we mean by the notation $\E\left[\e{\beta\tau} J_\beta\left(\X{x,\alpha}_\tau,\alpha^{\tau}\right)\right]$ the quantity
 \begin{equation*}
   \int_{\Om} {\e{\beta\tau(\om)} J_\beta\left(\X{x,\alpha}_\tau(\om),\alpha^{\tau(\om),\om}\right) \,\P(d\om)}.
 \end{equation*}
\end{proposition}

\begin{proof}
 Lemma~\ref{lem:cond} yields
 \begin{equation*}
  \E\left[\int_\tau^{+\infty} {\e{\beta s} f\left(\X{x,\alpha}_s,\alpha_s\right) ds} \, \Big| \,  \Fc_{\tau}\right](\om) =
  \e{\beta\tau(\om)} J_\beta\left(\X{x,\alpha}_{\tau}(\om),\alpha^{\tau(\om),\om}\right),\quad \P(d\om)-\as
 \end{equation*}
 The conclusion follows from conditioning by $\Fc_\tau$ in the cost function.
\end{proof}

We are now in a position to give the DPP, which is the keytone of our analysis of this optimal stochastic control problem.

\begin{theorem}\label{thm:DPP}
 Under Hypotheses~\ref{hyp:population}, \ref{hyp:cost}~(i) and \ref{hyp:bounded-inf}, it holds for all $x\in\ZZ_+$ and $\tau$ bounded stopping time,
 \begin{equation*}
  v_{\beta}(x) = \inf_{\alpha\in\Ac} \E\left[\int_0^{\tau} {\e{\beta s} f\left(\X{x,\alpha}_s,\alpha_s\right) ds} + \e{\beta\tau} v_\beta\left(\X{x,\alpha}_{\tau}\right)\right].
 \end{equation*}
\end{theorem}

\begin{proof}
 In view of Proposition~\ref{prop:costcond}, it is clear that
 \begin{equation*}
  J_{\beta}\left(x,\alpha\right) \geq \E\left[\int_0^{\tau} {\e{\beta s} f\left(\X{x,\alpha}_s,\alpha_s\right) ds} + \e{\beta\tau} v_\beta\left(\X{x,\alpha}_{\tau}\right)\right].
 \end{equation*}
 Hence, we conclude that
 \begin{equation*}
  v_{\beta}(x) \geq \inf_{\alpha\in\Ac} \E\left[\int_0^{\tau} {\e{\beta s} f\left(\X{x,\alpha}_s,\alpha_s\right) ds} + \e{\beta\tau} v_\beta\left(\X{x,\alpha}_{\tau}\right)\right].
 \end{equation*}
 The reverse inequality is more difficult to derive. The idea is to construct a convenient control by concatenation of an arbitrary control on $[0,\tau]$ and a well-chosen control on $[\tau,\infty)$. We start by working with a simple stopping time, \textit{i.e.}, taking values in a countable subset on $\R_+$. Given $\varepsilon>0$, let $\alpha^y$ be an $\varepsilon$--optimal control for the discount factor $\beta$ and the initial state $y\in\ZZ_+$, \textit{i.e.},
 \begin{equation*}
  v_\beta(y) \geq J_\beta\left(y,\alpha^y\right) - \varepsilon.
 \end{equation*}
 Let also $\alpha$ be an arbitrary admissible control. Define the control $\tilde{\alpha}$ by
 \begin{equation*}
  \tilde{\alpha}_s(\om) = 
    \begin{cases}
     \alpha_s(\om), & \text{if } 0\leq s\leq \tau(\om), \\
     \alpha^{\X{x,\alpha}_{\tau}(\om)}_{s-t}(Q_1^{-\tau(\om)}(\om),Q_2^{-\tau(\om)}(\om)), & \text{if } s>\tau(\om),
    \end{cases}
 \end{equation*}
 where, for $\mu\in\Md$, $\mu^{-t}$ denotes the image measure of $\restriction{\mu}{[t,\infty)}$ by the map
 $(s,z)\in[t,+\infty)\x\R_+\mapsto (s-t,z)\in\R_+\x\R_+$. Since $\tau$ is a simple stopping time, it is clear that the process $\tilde{\alpha}$ is predictable and thus an admissible control.
 By definition, $\tilde{\alpha}^{\tau(\om),\om}$ coincides with $\alpha^{\X{x,\alpha}_{\tau}(\om)}$ and 
 we have
 \begin{align*}
  \E\left[\e{\beta\tau} v_\beta\left(\X{x,\alpha}_\tau\right)\right] 
    & \geq \E\left[\e{\beta\tau} J_\beta\left(\X{x,\alpha}_\tau,\alpha^{\X{x,\alpha}_\tau}\right)\right] - \varepsilon \E\left[\e{\beta\tau}\right] \\
    & \geq \E\left[\e{\beta\tau} J_\beta\left(\X{x,\alpha}_\tau,\tilde{\alpha}^{\tau}\right)\right] - \varepsilon \e{\beta\lVert\tau\rVert_{\infty}}.
 \end{align*}
 Then it follows from Proposition~\ref{prop:costcond} that
 \begin{align*}
  \E\left[\int_0^{\tau} {\e{\beta s} f\left(\X{x,\alpha}_s,\alpha_s\right) ds} + \e{\beta\tau} v_\beta\left(\X{x,\alpha}_{\tau}\right)\right] 
    & \geq J_\beta\left(x,\tilde{\alpha}\right) - \varepsilon \e{\beta\lVert\tau\rVert_{\infty}} \\
    & \geq v_\beta(x) - \varepsilon \e{\beta\lVert\tau\rVert_{\infty}}.
 \end{align*}
 Sending $\varepsilon$ to zero and taking the infimum over $\alpha\in\Ac$, we conclude that
 \begin{equation*}
  v_{\beta}(x) \leq \inf_{\alpha\in\Ac} \E\left[\int_0^{\tau} {\e{\beta s} f\left(\X{x,\alpha}_s,\alpha_s\right) ds} + \e{\beta\tau} v_\beta\left(\X{x,\alpha}_{\tau}\right)\right].
 \end{equation*} 
 To conclude, it remains to derive the same result for any bounded stopping time. To this end, we consider $(\tau_n)_{n\in\N}$ a decreasing sequence of simple stopping times converging to $\tau$, \textit{e.g.},
 \begin{equation*}
  \tau_n = \sum_{i\geq 0} \frac{i+1}{2^n} \mathds{1}_{\frac{i}{2^n}<\tau\leq \frac{i+1}{2^n}}.
 \end{equation*} 
 In view of the above, we have for all $n\in\N$,
 \begin{equation*}
  \E\left[\int_0^{\tau_n} {\e{\beta s} f\left(\X{x,\alpha}_s,\alpha_s\right) ds} + \e{\beta\tau_n} v_\beta\left(\X{x,\alpha}_{\tau_n}\right)\right] 
     \geq v_\beta(x) - \varepsilon \e{\beta\lVert\tau_n\rVert_{\infty}}.
 \end{equation*}
 Since $v_\beta$ is bounded by assumption, the conclusion follows by applying the dominated convergence theorem as $n$ tends to infinity.
\end{proof}

From the DPP, we derive in the next theorem the HJB equation corresponding to the minimization problem.

\begin{theorem}
  \label{thm:HJB}
  Under Hypotheses~\ref{hyp:population}, \ref{hyp:cost}~(i) and \ref{hyp:bounded-inf}, it holds 
  \begin{equation}\label{eq:HJB}
   \beta v_{\beta}(x) + \inf_{a\in A} {\left\{\Lc^a {v_\beta} (x) + f(x,a)\right\}} = 0, \quad \forall\,x\in\N.
  \end{equation}
\end{theorem}

This result can be deduced from the DPP by classical arguments. We give its proof for the sake of completeness.

\begin{proof}
 Let $x\in\N$ be fixed. Denote by $\tau_h$ the stopping time given by
 \begin{equation*}
  \tau_h := \inf{\left\{ s\geq 0;\ Q_1([0,s]\x[0,\bmax x]) + Q_2([0,s]\x[0,\dmax_x]) \neq 0  \right\}}\wedge h,
 \end{equation*}
 for some $h>0$. 
 First, we notice that
 \begin{multline*}
  \e{\beta \tau_h} v_\beta\left(\X{x,\alpha}_{\tau_h}\right) = v_{\beta}(x) + \int_0^{\tau_h} {\beta \e{\beta\theta} v_\beta(x) \,d\theta} \\
  \begin{aligned}
   & + \int_{[0,\tau_h]\x\R_+} {\sum_{k\geq 1} {\left(\e{\beta\theta} v_\beta\left(x+k\right) - \e{\beta\theta} v_\beta\left(x\right)\right) \mathds{1}_{I_k\left(x,\alpha_{\theta}\right)}(z) } \,Q_1(d\theta, dz)} \\
   & + \int_{[0,\tau_h]\x\R_+} {\left( \e{\beta\theta} v_\beta\left(x-1\right) - \e{\beta\theta} v_\beta\left(x\right)\right) \mathds{1}_{[0,d\left(x,\alpha_{\theta}\right)]}(z) \,Q_2(d\theta, dz)}.  
  \end{aligned}
 \end{multline*}
 Taking expectation, it follows from the boundedness of $v_\beta$ that
 \begin{equation*}
  \E\left[\e{\beta \tau_h} v_\beta\left(\X{x,\alpha}_{\tau_h}\right)\right] = v_{\beta}(x)  + \E\left[\int_0^{\tau_h} {\e{\beta\theta}\left(\beta v_\beta(x) + \Lc^{\alpha_\theta} {v_\beta}(x)\right) d\theta}\right].  
 \end{equation*}
 Using the DPP, we deduce that
 \begin{equation*}
  \inf_{\alpha\in\Ac} {\left\{\E\left[\int_0^{\tau_h} {\e{\beta\theta} \left(\beta v_\beta(x) +  \Lc^{\alpha_\theta} {v_\beta}(x) + f\left(x,\alpha_{\theta}\right) \right) d\theta}\right]\right\}} = 0.  
 \end{equation*}
 Since constant controls are admissible, the identity above is equivalent to 
 \begin{equation*}
  \left(\beta v_\beta(x) +  \inf_{a\in A} {\left\{\Lc^{a} {v_\beta}(x) + f\left(x,a\right)\right\}}\right) \E\left[\frac{1}{h}\int_0^{\tau_h} {\e{\beta\theta} \,d\theta}\right] = 0.
 \end{equation*}
 Sending $h$ to zero, we conclude by the dominated convergence theorem.
\end{proof}

 \begin{theorem}
   \label{thm:alpha-opt-Markov-sup}
  Under Hypotheses~\ref{hyp:population}, \ref{hyp:cost}~(i), \ref{hyp:main} and \ref{hyp:bounded-inf}, there exists $\alpha_{\beta}\in\Ac_M$ such that
  \begin{equation*}
   \inf_{a\in A} {\left\{\Lc^a {v_\beta} (x) + f(x,a)\right\}} = \Lc^{\alpha_\beta(x)} {v_\beta} (x) + f(x,\alpha_{\beta}(x)).
  \end{equation*}
  In addition, the Markov control $\alpha_{\beta}$ is optimal for the $\beta$--discounted minization problem on infinite horizon, that is, $v_\beta
  = J_\beta(\cdot,\alpha_{\beta})$. 
 \end{theorem}

\begin{proof}
 Let us show first that the map $a\mapsto \Lc^a {v_\beta} (x) + f(x,a)$  is continuous on $A$ for all $x\in\ZZ_+$. In view of the
 continuity of $\gamma(x,\cdot)$ , $f(x,\cdot)$, $p_k(x,\cdot)$ and $d(x,\cdot)$, 
 it is clearly enough to prove that
 \begin{equation*}
  \lim_{K\to\infty} {\sup_{a\in A} {\left\{\sum_{k\geq K+1} {v_\beta\left(x+k\right) p_k\left(x,a\right)}\right\}}} = 0.
 \end{equation*} 
 The former follows from 
 \begin{align*}
  \sum_{k\geq K+1} {v_\beta\left(x+k\right) p_k\left(x,a\right)} \leq \lVert v_\beta\rVert_{\infty} \sum_{k\geq K+1} {p_k\left(x,a\right)} \leq \frac{M \lVert v_\beta\rVert_{\infty}}{K+1},
 \end{align*}
 where the constant $M$ comes from Hypothesis~\ref{hyp:population}~(iv).
 Since $A$ is compact, we deduce that there exists
 $\alpha_{\beta}\in\Ac_M$ such that
 \begin{equation*}
  \inf_{a\in A} {\left\{\Lc^a {v_\beta} (x) + f(x,a)\right\}} = \Lc^{\alpha_\beta(x)} {v_\beta} (x) + f(x,\alpha_{\beta}(x)).
 \end{equation*}
 The rest of the proof relies on a classical verification theorem. 
 Using the boundedness of $v_\beta$, we derive as in the proof of Theorem~\ref{thm:HJB} that
 \begin{equation*} 
  \E\left[\e{\beta T} v_\beta\left(\X{x,\alpha_{\beta}}_T\right)\right] = v_\beta(x)
   + \E\left[\int_0^T { \e{\beta \theta} \left(\beta v_\beta\left(\X{x,\alpha_{\beta}}_\theta\right) + \Lc^{\alpha_{\beta}} v_\beta\left(\X{x,\alpha_{\beta}}_\theta\right)\right) d\theta }\right].
 \end{equation*}
 Since $v_\beta$ is a solution of the HJB equation, we deduce that
 \begin{equation*}
  \E\left[\int_0^T {\e{\beta\theta} f^{\alpha_\beta}\left(\X{x,\alpha_{\beta}}_\theta\right)} \,d\theta + \e{\beta T} v_\beta\left(\X{x,\alpha_\beta}_T\right)\right] = v_\beta(x).
 \end{equation*}
 In particular, we have 
 \begin{equation*}
  \E\left[\int_0^T {\e{\beta\theta} f^{\alpha_\beta}\left(\X{x,\alpha_{\beta}}_\theta\right)} \,d\theta\right] \leq v_\beta(x).
 \end{equation*}
 Sending $T$ to infinity, it follows from the monotone convergence theorem that
 \begin{equation*}
  \E\left[\int_0^{+\infty} {\e{\beta\theta} f^{\alpha_\beta}\left(\X{x,\alpha_\beta}_\theta\right) d\theta}\right] \leq v_\beta(x).
 \end{equation*}
 We conclude that $\alpha_\beta$ is an optimal control.
\end{proof}

\begin{remark}
 Our analysis can be extended to deal with the so-called continuous-time Markov decision processes (CTMDP) (see, \eg,~\cite{guo-hernandez-09,piunovskiy-zhang-13}). As such, this section presents a new methodology to derive the HJB equation and the existence of an optimal Markov control. Even though they can be weakened, our assumptions are rather strong compared to the literature on the CTMDP. The benefit of this approach is to unify the treatment of stochastic control problems for jump processes and diffusion processes. Indeed, by a simple analogy which consists in giving to the Brownian motion the role played by the Poisson random measures, we can follow the same arguments to derive the DPP corresponding to a stochastic control problem on diffusion processes. See~\cite{claisse-talay-tan-16} for more details.\\
\end{remark}

 \subsection{Maximisation problem}
 
  In this section, we derive the corresponding results for the maximization problem. 
  They are stated under the assumption that the value function $w_\beta$ is bounded, which is proved for $\beta<\lambda_*$ in Lemma~\ref{lem:bounded} below. Note that this is a delicate issue in this setting.
  
 \begin{hypothesis}\label{hyp:bounded-sup}
  The value function $w_\beta$ is bounded on $\ZZ_+$.
 \end{hypothesis}

  \begin{theorem}
   \label{thm:alpha-opt-Markov}
  Under Hypotheses~\ref{hyp:population}, \ref{hyp:cost}~(i), \ref{hyp:main} and \ref{hyp:bounded-sup},
  it holds
  \begin{equation}\label{eq:HJB_sup}
   \beta w_{\beta}(x) + \sup_{a\in A} {\left\{\Lc^a {w_\beta} (x) + f(x,a)\right\}} = 0, \quad \forall\,x\in\N,
  \end{equation}
  and there exists $\alpha_{\beta}\in\Ac_M$ such that
  \begin{equation*}
   \sup_{a\in A} {\left\{\Lc^a {w_\beta} (x) + f(x,a)\right\}} = \Lc^{\alpha_\beta(x)} {w_\beta} (x) + f(x,\alpha_{\beta}(x)).
  \end{equation*}
  If we assume further that
  \begin{equation}\label{eq:verification}
   \lim_{T\to+\infty} {e^{\beta T} \E_x\left[w_\beta\left(\X{\alpha_\beta}_T\right)\right]} = 0, \quad \forall\,x\in\N,
  \end{equation}  
  then the Markov control $\alpha_{\beta}$ is optimal for the $\beta$--discounted maximization problem on infinite horizon, that is, $w_\beta
  = J_\beta(\cdot,\alpha_{\beta})$. 
 \end{theorem}

\begin{proof}
  The proof relies on the arguments developed in the previous section. Indeed, it follows by repeating the argument of Theorem~\ref{thm:DPP} that the corresponding DPP holds, \textit{i.e}, for any bounded stopping time $\tau$,
  \begin{equation*}
  w_{\beta}(x) = \sup_{\alpha\in\Ac} \E\left[\int_0^{\tau} {\e{\beta s} f\left(\X{x,\alpha}_s,\alpha_s\right) ds} + \e{\beta\tau} w_\beta\left(\X{x,\alpha}_{\tau}\right)\right].
 \end{equation*}
 Then we deduce as in Theorem~\ref{thm:HJB} that the value function satisfies the HJB equation~\eqref{eq:HJB_sup}. To conclude, it remains to adapt the arguments of Theorem~\ref{thm:alpha-opt-Markov}. First, the existence of $\alpha_\beta$ follows from the continuity of the map $a\mapsto \Lc^a {w_\beta} (x) + f(x,a)$, which results from the boundedness of the value function as before.
 Second, the proof of the verification theorem needs to be slighty modified.
 Since $w_\beta$ is a solution of the HJB equation, we have 
 \begin{equation*}
  \E\left[\int_0^T {\e{\beta\theta} f^{\alpha_\beta}\left(\X{x,\alpha_{\beta}}_\theta\right)} \,d\theta + \e{\beta T} w_\beta\left(\X{x,\alpha_\beta}_T\right)\right] = w_\beta(x).
 \end{equation*}
 Sending $T$ to infinity and using~\eqref{eq:verification}, it follows from the monotone convergence theorem that
 \begin{equation*}
  \E\left[\int_0^{+\infty} {\e{\beta\theta} f^{\alpha_\beta}\left(\X{x,\alpha_\beta}_\theta\right) d\theta}\right] = w_\beta(x).
 \end{equation*}
 We conclude that $\alpha_\beta$ is an optimal control.
\end{proof}

\section{Quasi-stationary distributions}
\label{sec:QSD}

We fix $\alpha\in\mathcal{A}_M$ in all this section. Our goal is to study the QSD of the Markov process $X^\alpha$, that is the
probability measure $\pi^\alpha$ on $\N$ such that
$$
\P_{\pi^\alpha}(X^{\alpha}_t\in\Gamma\mid t<\tau^{\alpha})=\pi^\alpha(\Gamma),
$$
where $\P_{\pi^\alpha}=\sum_{k\geq 1}\pi^\alpha_k\P_k$, and the associated absorption rate $\lambda^\alpha>0$
defined by
\begin{equation}
  \label{eq:def-lambda}
  \P_{\pi^\alpha}(t<\tau^\alpha)=e^{-\lambda^\alpha t},\quad\forall\,t\geq 0.
\end{equation}
Below, we prove Theorem~\ref{thm:QSD-short-version} and auxiliary estimates needed for Proposition~\ref{thm:QSD} and for the
proof of our main results. In particular, throughout this section, we work under Hypotheses~\ref{hyp:population} and~\ref{hyp:qsd}.

Since the process $X^\alpha$ is Markov, a.s.\ absorbed in finite time in 0, and satisfies
$$
\P_x(t<\tau^\alpha)>0,\quad\forall\,x\geq 1\text{ and }t>0,
$$
we can apply the general criterion of~\cite[Thm.\,2.1]{champagnat-villemonais-15} to prove the existence and uniqueness of a QSD
$\pi^\alpha$ and the existence of constants $C$ and $\gamma>0$ such that, for all probability measure $\mu$ on $\N$,
\begin{equation}
  \label{eq:cv-expo}
  \left\|\P_\mu(X^{\alpha}\in\cdot\mid t<\tau^{\alpha})-\pi^\alpha\right\|_{TV}\leq Ce^{-\gamma t},\quad\forall t\geq 0.
\end{equation}
These three properties are equivalent to the following condition, which also implies several other properties including those of
Proposition~\ref{thm:QSD}.

\paragraph{Condition~(A)}
There exists a probability measure $\nu$ on $\N$ such that
\begin{itemize}
\item[(A1)] there exists $t_0,c_1>0$ such that for all $x\in\N$,
  $$
  \P_x(X^\alpha_{t_0}\in\cdot\mid t_0<\tau^\alpha)\geq c_1\nu(\cdot);
  $$
\item[(A2)] there exists $c_2>0$ such that for all $x\in\N$ and $t\geq 0$,
  $$
  \P_\nu(t<\tau^\alpha)\geq c_2\P_x(t<\tau^\alpha).
  $$
\end{itemize}

\begin{proof}[Proof of Theorem~\ref{thm:QSD-short-version} and Proposition~\ref{thm:QSD}]
  The next lemma, proved at the end of this section, allows to check Conditions~(A1) and~(A2) with $\nu=\delta_1$.
  
  \begin{lemma}
    \label{lem:hypcv}
    For all $\alpha\in\Ac_M$, one has:\\
    \rmi there exists $t_0,c_1>0$ such that for all $x\in\N$,
    \begin{equation*}
      \P_x\left(\X{\alpha}_{t_0}=1 \mid t_0<\tau^{\alpha}\right) \geq c_1;
    \end{equation*}
    \rmii there exists $c_2>0$ such that for all $x\in\N$ and $t\geq0$,
    \begin{equation*}
      \P_{1}\left(t<\tau^{\alpha}\right) \geq c_2\P_{x}\left(t<\tau^{\alpha}\right).
    \end{equation*}
    In addition, the constants $t_0$, $c_1$ and $c_2$ do not depend on $\alpha\in\mathcal{A}_M$.
  \end{lemma}
  
  By~\cite[Thm.\,2.1]{champagnat-villemonais-15}, these properties directly imply~\eqref{eq:cv-expo} and hence
  Theorem~\ref{thm:QSD-short-version} with $\gamma=-\log(1-c_1c_2)/t_0$ and $C=2/(1-c_1c_2)$. Since $t_0$, $c_1$ and $c_2$ do not
  depend on $\alpha$, so do $C$ and $\gamma$.

  The existence of the function $\eta^\alpha$ of Proposition~\ref{thm:QSD}, limit of $\P_\cdot(t<\tau^\alpha)e^{\lambda^\alpha t}$
  and eigenfunction of $\mathcal{L}^\alpha$, is given by~\cite[Prop.\,2.3]{champagnat-villemonais-15}. In addition, the proof of
  Proposition 2.3 in~\cite[Sec.5.2]{champagnat-villemonais-15} implies that~\eqref{eq:etadef} holds with the constant
  $C'=C\kappa^\alpha$, where
  $$
  \kappa^\alpha:=\sup_{t\geq 0,\,\mu\in {\cal M}_1(\N)} e^{\lambda^\alpha t}\P_\mu(t<\tau^\alpha),
  $$
  where $\mathcal{M}_1(\N)$ is the set of probability measures on $\N$. It only remains to check that $\kappa^\alpha$ is bounded as a
  function of $\alpha\in\mathcal{A}_M$. It follows from~\cite[Remark~1]{champagnat-villemonais-15} that
  \begin{equation*}
    \kappa^\alpha\leq\inf_{s>0}\exp\left(\lambda^\alpha s+\frac{C e^{(\lambda^\alpha-\gamma)s}}{1-e^{-\gamma s}}\right),
  \end{equation*}
  where $C$ and $\gamma$ are the constants of Theorem~\ref{thm:QSD-short-version}. So the proof will be completed if we check that
  $\lambda^\alpha$ is uniformly bounded w.r.t.\ $\alpha$. This follows from the general fact that, for any Markov process in $\ZZ_+$
  absorbed at 0, the extinction time is larger than an exponential random variable of parameter the supremum of the extinction rate
  in the population from any state in $\N$. In our branching process, this supremum is $d^\alpha_1\leq \bar{d}_1$. Hence, it follows
  from~\eqref{eq:def-lambda} that $\lambda^\alpha\leq \bar{d}_1$, which ends the proof of Proposition~\ref{thm:QSD}.
\end{proof}

In order to prove Lemma~\ref{lem:hypcv}, we need to prove that the controlled process \textit{comes down from
  infinity}~\cite{vandoorn-91,champagnat-villemonais-15}. The next lemma, based on similar arguments as
in~\cite{champagnat-villemonais-PNM-15}, is a preliminary step for this.

\begin{lemma}
  \label{lem:downfrominf}
  For all $\lambda>0$, there exists $x_\lambda\in\N$ such that
  \begin{equation*}
    \sup_{\substack{x\geq x_\lambda \\ \alpha\in\Ac_M}} {\E_x\left[\e{\lambda \zeta^{\alpha}_{\lambda}}\right]} < +\infty,
  \end{equation*}
  where $ \zeta^{\alpha}_{\lambda}$ denotes the first hitting time of $\{x_\lambda\}$ by $\X{\alpha}$, that is,
  \begin{equation*}
     \zeta^{\alpha}_{\lambda} := \inf {\left\{t\geq0:\ \X{\alpha}_t = x_\lambda\right\}}.
  \end{equation*}
\end{lemma}

\begin{proof}
 Let $\psi:\ZZ_+\rightarrow\R_+$ be given by
  \begin{equation*}
    \psi(x) := \sum_{y=1}^{x} \frac{1}{y^{1+\frac{\epsilon}{2}}},\quad \forall\,x\in\ZZ_+,
  \end{equation*}  
  where $\epsilon>0$ comes from Hypothesis~\ref{hyp:qsd}~(i). We start by proving that, for all $\lambda>0$, there exists $x_\lambda\in\N$ such that
  \begin{equation}\label{eq:lemmastep1}
    \Lc^{\alpha} {\psi} (x) \leq -\lambda\psi(x), \quad \forall\,x\geq x_\lambda, \alpha\in\Ac_M.
  \end{equation}
  It follows from Hypotheses~\ref{hyp:population} and \ref{hyp:qsd} that
  \begin{align*}
    \Lc^{\alpha} {\psi} (x) 
    &\leq \dmin x^{1+\epsilon} \left(\psi(x-1)-\psi(x)\right) + \bmax x \sum_{k=1}^{+\infty} {\left(\psi(x+k)-\psi(x)\right)p^{\alpha}_k(x)} \\
    &\leq -\dmin x^{\frac{\epsilon}{2}} + \bmax x \sum_{k=1}^{+\infty} {\left(\sum_{y=x+1}^{x+k} \frac{1}{y^{1+\frac{\epsilon}{2}}}\right)p^{\alpha}_k(x)}.
  \end{align*}
  Further, one has
  \begin{align*}
    \sum_{k=1}^{+\infty} {\left(\sum_{y=x+1}^{x+k} {\frac{1}{y^{1+\frac{\epsilon}{2}}}}\right)p^{\alpha}_k(x)} 
    \leq\frac{1}{x^{1+\frac{\epsilon}{2}}}\sum_{k=1}^{+\infty} kp^\alpha_k(x)\leq \frac{M}{x^{1+\frac{\epsilon}{2}}},
  \end{align*}
  where $M>0$ comes from Hypothesis~\ref{hyp:population}~(iv). 
 Hence, we deduce that there exists $x_1$ such that
  \begin{equation*}
    \Lc^{\alpha} {\psi} (x) \leq - \frac{\dmin}{2} x^{\frac{\epsilon}{2}}, \quad\forall\,x\geq x_1.
  \end{equation*}
  Now, given $\lambda>0$, we take $x_\lambda$ as the smallest integer not less than $x_1$ and $(\frac{2\lambda\psi(\infty)}{\dmin})^{\frac{2}{\epsilon}}$, where $\psi(\infty)$ denotes the limit of $\psi(x)$.
  Then one has for all $x\geq x_\lambda$,
  \begin{equation*}
    \Lc^{\alpha} {\psi} (x) \leq - \frac{\dmin}{2} {x_\lambda}^{\frac{\epsilon}{2}} \leq - \frac{\dmin{x_\lambda}^{\frac{\epsilon}{2}}}{2\psi(\infty)} \psi(x) \leq - \lambda  \psi(x).
  \end{equation*}
  
  We are now in a position to complete the proof. As already observed, we have
  \begin{multline*}
    \e{\lambda s} \psi\left(\X{x,\alpha}_s\right) = \psi(x) + \int_0^s {\e{\lambda \theta}
      \lambda\psi\left(\X{x,\alpha}_\theta\right) d\theta}\\
    \begin{aligned}
      & + \int_{(t,s]\x\R_+}  { \e{\lambda \theta}\sum_{k=1}^{+\infty} {\left(\psi\left(\X{x,\alpha}_{\theta-} + k\right) -
              \psi\left(\X{x,\alpha}_{\theta-}\right)\right) \mathds{1}_{I_k^\alpha\left(\X{x,\alpha}_{\theta-}\right)}(z)}
         \,Q_1(d\theta, dz)}\\
      & + \int_{(t,s]\x\R_+}  { \e{\lambda \theta}\left(\psi\left(\X{x,\alpha}_{\theta-} - 1\right) -
            \psi\left(\X{x,\alpha}_{\theta-}\right)\right)
          \mathds{1}_{\left[0,d^\alpha\left(\X{x,\alpha}_{\theta-}\right)\right)}(z) \,Q_2(d\theta, dz)}.
    \end{aligned}
  \end{multline*}
  We localize the relation above by $\zeta^{\alpha}_{\lambda}$ and we take the expectation. Since $\psi$ is bounded, we obtain
  \begin{equation*}
    \E_x\left[\e{\lambda \zeta^{\alpha}_{\lambda}\wedge s} \psi\left(\X{\alpha}_{\zeta^{\alpha}_{\lambda}\wedge s}\right)\right] = \psi(x) + \E_x\left[ \int_0^{\zeta^{\alpha}_{\lambda}\wedge s} {\e{\lambda \theta} \left(\lambda\psi\left(\X{\alpha}_\theta\right) + \Lc^{\alpha_\theta}{\psi}\left(\X{\alpha}_\theta\right)\right) d\theta}\right].
  \end{equation*}
  It follows from the inequality~\eqref{eq:lemmastep1} that
  \begin{equation*}
   \E_x\left[\e{\lambda \zeta^{\alpha}_{\lambda}\wedge s} \psi\left(\X{\alpha}_{\zeta^{\alpha}_{\lambda}\wedge s}\right)\right] \leq \psi(x),\quad \forall\,x\geq x_\lambda, \alpha\in\Ac_M.
  \end{equation*}
  Hence, we deduce that
  \begin{equation*}
    \E_x\left[\e{\lambda \zeta^{\alpha}_{\lambda}\wedge s}\right] \leq \frac{\psi(\infty)}{\psi(x_\lambda)}, \quad \forall\,x\geq x_\lambda, \alpha\in\Ac_M.
  \end{equation*}
  The conclusion follows from the monotone convergence theorem.
\end{proof}

\begin{proof}[Proof of Lemma~\ref{lem:hypcv}]
  This proof uses a similar method as in~\cite[Thm.\,4.1]{champagnat-villemonais-15}. Let $\lambda=\dmax_1+1$ and $x_\lambda\in\N$ be
  given by Lemma~\ref{lem:downfrominf} such that
  \begin{equation*}
    C_1:=\sup_{\substack{x\geq x_\lambda \\ \alpha\in\Ac_M}} {\E_x\left[\e{\lambda \zeta^{\alpha}_{\lambda}}\right]} < +\infty.
  \end{equation*}
  We start by proving~(ii). The Markov property yields, for all $x\in\N$,
  \begin{equation*}
    \P_1\left(\X{\alpha}_1=x\right)\P_x\left(t<\tau^{\alpha}\right) \leq \P_1\left(t+1<\tau^{\alpha}\right) \leq \P_1\left(t<\tau^{\alpha}\right).
  \end{equation*}
  Hence, if we denote 
  \begin{equation*}
    \frac{1}{C_2}:=\inf_{\substack{\alpha\in\Ac_M \\1\leq y,z\leq x_\lambda}} {\P_y\left(\X{\alpha}_1 = z\right)},
  \end{equation*}
  we deduce that for all $1\leq x\leq x_\lambda$,
  \begin{equation}
   \label{eq:lem}
    \P_x\left(t<\tau^{\alpha}\right)\leq C_2 \P_1\left(t<\tau^{\alpha}\right).
  \end{equation}
  Note that $C_2$ is well-defined in view of Hypothesis~\ref{hyp:qsd}~(ii).
  Further, we have for all $x\geq x_\lambda$,
  \begin{equation*}
    \P_x\left(t<\tau^{\alpha}\right) = \P_x\left(t <  \zeta^{\alpha}_{\lambda}\right) + \P_x\left( \zeta^{\alpha}_{\lambda} < t < \tau^{\alpha}\right).
  \end{equation*}
  By Markov's inequality, the first term on the r.h.s.\ above can be bounded as follows:
  \begin{equation*}
    \P_x\left(t <  \zeta^{\alpha}_{\lambda}\right) \leq C_1 \e{-\lambda t} \leq C_1 \P_1\left(t<\tau^{\alpha}\right),
  \end{equation*}
  where the last inequality comes from $\lambda>\dmax_1$. For the second term, the Markov property and the inequality~\eqref{eq:lem} yield
  \begin{align*}
    \P_x\left( \zeta^{\alpha}_{\lambda} < t < \tau^{\alpha}\right) 
    & = \int_0^t {\P_{x_\lambda}\left(t-s < \tau^{\alpha}\right) \P_{x}\left( \zeta^{\alpha}_{\lambda}\in ds\right)} \\
    & \leq C_2 \int_0^t {\P_{1}\left(t-s < \tau^{\alpha}\right) \P_{x}\left( \zeta^{\alpha}_{\lambda}\in ds\right)}.
  \end{align*}
  Moreover, since $\lambda>\dmax_1$, we have
  \begin{equation*}
    \e{-\lambda s} \P_{1}\left(t-s < \tau^{\alpha}\right) \leq \inf_{x\geq 1} {\left\{\P_x\left(s < \tau^{\alpha}\right)\right\}}
    \P_{1}\left(t-s < \tau^{\alpha}\right) \leq \P_{1}\left(t < \tau^{\alpha}\right).
  \end{equation*}
  Hence, we obtain
  \begin{equation*}
    \P_x\left( \zeta^{\alpha}_{\lambda} < t < \tau^{\alpha}\right) 
     \leq C_2 \P_{1}\left(t < \tau^{\alpha}\right) \E_x\left[\e{\lambda \zeta^{\alpha}_{\lambda}}\right]
     \leq C_1 C_2 \P_{1}\left(t < \tau^{\alpha}\right).
  \end{equation*}
  This ends the proof of the assertion~(ii).
  
  Now we prove the assertion~(i). Let $\lambda$, $x_\lambda$, $C_1$ and $C_2$ be defined as above. For all $t\geq0$ and
  $\alpha\in\Ac_M$, we consider two cases: either $x\leq x_\lambda$,
  \begin{equation*}
   \P_{x}\left(\X{\alpha}_{t+1}=1\right) 
     \geq \P_{x}\left(\X{\alpha}_{1} = 1\right) \P_1\left(\X{\alpha}\text{ has no jump in }(0, t]\right) \\
     \geq \frac{1}{C_2}\e{-(\dmax_1 + \bmax)t},
  \end{equation*}
  or $x>x_\lambda$,
  \begin{align*}
    \P_{x}\left(\X{\alpha}_{t+1}=1\right) 
    & \geq \P_{x}\left( \zeta^{\alpha}_{\lambda}\leq t\right) \P_{x_\lambda}\left(\X{\alpha}_{1} = 1\right) \P_1\left(\X{\alpha}\text{ has no jump in }(0, t]\right) \\
    & \geq \left(1 - C_1\e{-\lambda t}\right)^+\frac{1}{C_2}\e{-(\dmax_1 + \bmax)t}.
  \end{align*}
  Hence, if we take $t_0>1$ such that $C_1\e{-\lambda (t_0 - 1)}\leq\frac{1}{2}$, we obtain
  \begin{equation*}
    \P_x\left(\X{\alpha}_{t_0}=1\right) \geq \frac{\e{-(\dmax_1 + \bmax)(t_0-1)}}{2 C_2}>0,\quad \forall\,x\in\N.
  \end{equation*}
  We finally notice that
  \begin{equation*}
    \P_x\left(\X{\alpha}_{t_0}=1\mid t_0<\tau^{\alpha}\right)\geq \P_x\left(\X{\alpha}_{t_0}=1\right),\quad \forall\,x\in\N,
  \end{equation*}
  which concludes the proof of Lemma~\ref{lem:hypcv}.
\end{proof}

We conclude this section with a proposition, which is used in the proof of Theorems~\ref{thm:main-inf} and \ref{thm:main-sup}.

\begin{proposition}\label{prop:minoration}
  For all $x\in\N$, it holds
  \begin{equation*}
   \inf_{\alpha\in\Ac_M} {\pi^{\alpha}(x)} > 0,\qquad \inf_{\alpha\in\Ac_M} {\eta^{\alpha}(x)} > 0.
  \end{equation*}
\end{proposition}

\begin{proof}
  Fix $x\in\N$. A straightforward extension of Lemma~\ref{lem:hypcv} yields\\
  \rmi there exists $t_{0},c_{1}>0$ such that for all $y\in\N$ and $\alpha\in\Ac_M$,
   \begin{equation*}
    \P_y\left(\X{\alpha}_{t_{0}} = x \mid t_{0}<\tau^{\alpha}\right) \geq c_{1};
   \end{equation*}
  \rmii there exists $c_{2}>0$ such that for all $y\in\N$, $t\geq0$ and $\alpha\in\Ac_M$,
   \begin{equation*}
    \P_{x}\left(t<\tau^{\alpha}\right) \geq c_{2}\P_{y}\left(t<\tau^{\alpha}\right).
   \end{equation*}
  Note that the constants $t_0$, $c_1$ and $c_2$ above depend on $x$.
  Then, we integrate the first inequality w.r.t. $\pi^{\alpha}$. It yields
  \begin{equation*}
   \pi^{\alpha}(x)\geq c_{1},\quad \forall\, \alpha\in\Ac_M.
  \end{equation*}
  Finally, we integrate the second inequality w.r.t. $\pi^{\alpha}$, multiply it by $e^{\lambda^{\alpha} t}$ and send $t$ to infinity. We obtain
  \begin{equation*}
   \eta^{\alpha}(x) \geq c_{2},\quad \forall\, \alpha\in\Ac_M. \qedhere
  \end{equation*}
\end{proof}

\section{Proof of the main results}
\label{sec:proof}

\subsection{Proof of Theorem~\ref{thm:main-inf}}\label{sec:proofthmmain}
  
 We begin with a premilinary result, which is a refinement of Proposition~\ref{prop:controlqsd}.

\begin{proposition}\label{prop:main}
  Under Hypotheses~\ref{hyp:population}, \ref{hyp:qsd} and \ref{hyp:cost}, there exists $C>0$ such that, for all $\alpha\in\Ac_M$
  satisfying $\lambda^\alpha>\beta$, for all $x\in\N$,
  \begin{equation*}
    \left|J_{\beta}(x,\alpha) - \frac{\pi^{\alpha}(f^{\alpha})}{\lambda^{\alpha}-\beta}\eta^{\alpha}(x)\right|\leq C.
  \end{equation*}
  In addition, under Hypotheses~\ref{hyp:main} and \ref{hyp:main-inf}, for all $\alpha\in\Ac_M$ such that 
  $\lambda^\alpha\leq\beta$, it holds $J_{\beta}(x,\alpha)=+\infty$ for all $x\in\N$.
\end{proposition}  

\begin{proof}
First we observe that 
\begin{multline*}
  \left|\E_x\left[f^{\alpha}\left(\X{\alpha}_t\right)\right] -  \pi^{\alpha}(f^{\alpha})\eta^{\alpha}(x)\e{-\lambda^{\alpha} t} \right| \\
  \leq \eta^{\alpha}(x)\e{-\lambda^{\alpha} t} \left|\E_x\left[f^{\alpha}\left(\X{\alpha}_t\right) \mid
      t<\tau^{\alpha}\right] -  \pi^{\alpha}(f^{\alpha}) \right|  \\
  +  \lVert f\rVert_{\infty} \left|\P_x\left(t<\tau^{\alpha}\right) - \eta^{\alpha}(x)\e{-\lambda^{\alpha} t} \right|.
\end{multline*}
On the one hand, we deduce from Theorem~\ref{thm:QSD-short-version} that
\begin{align*}
  \left|\E_x\left[f^{\alpha}\left(\X{\alpha}_t\right) \mid t<\tau^{\alpha}\right] -  \pi^{\alpha}(f^{\alpha}) \right|
    & \leq \lVert f\rVert_{\infty} \TV{\P_x\left(\X{\alpha}_t\in\cdot \mid t<\tau^{\alpha}\right) -\pi^{\alpha}} \\
    & \leq C \lVert f\rVert_{\infty} \e{-\gamma t}.
\end{align*}
On the other hand, it follows from Proposition~\ref{thm:QSD} that for all $x\in\N$ and $\alpha\in\Ac_M$,
\begin{equation*}
 \eta^{\alpha}(x)\leq C' + 1 ~~ \text{and} ~~  \left|\P_x\left(t<\tau^{\alpha}\right) - \eta^{\alpha}(x)\e{-\lambda^{\alpha} t} \right|\leq C'e^{-(\lambda^{\alpha}+\gamma)t}.
\end{equation*}
Hence, we deduce that
\begin{equation}
 \label{eq:prop_main}
  \left|\E_x\left[f^{\alpha}\left(\X{\alpha}_t\right)\right] -  \pi^{\alpha}(f^{\alpha})\eta^{\alpha}(x)\e{-\lambda^{\alpha} t}
  \right| \leq C_0  \e{-(\lambda^\alpha + \gamma)t}
\end{equation}
where $C_0 := (C + C' + CC') \lVert f\rVert_{\infty}$. Hence, for all $x\in\N$ and $\alpha\in\Ac_M$ such that $\lambda^\alpha > \beta$,
\begin{align*}
  \left|J_{\beta}(x,\alpha) - \frac{\pi^{\alpha}(f^{\alpha})}{\lambda^{\alpha}-\beta}\eta^{\alpha}(x)\right|
  & \leq \int_0^{+\infty} \!\!\!{\e{\beta t} \left|\E_x\left[f^{\alpha}\left(\X{\alpha}_t\right)\right] -
      \pi^{\alpha}(f^{\alpha})\eta^{\alpha}(x)\e{-\lambda^{\alpha} t} \right| dt} \\
  & \leq C_0 \int_0^{+\infty} { \e{-(\lambda^\alpha + \gamma - \beta)t} \,dt}
  \leq \frac{C_0}{\gamma}. 
\end{align*}
 This ends the proof of the first assertion. As for the second one, we observe first that,
 under Hypotheses~\ref{hyp:main} and \ref{hyp:main-inf}, 
 there exists $x_0$ such that $\inf_{a\in A} \{f(x_0,a)\} > 0$. 
 Together with Proposition~\ref{prop:minoration}, it yields for all $x\in\N$
 \begin{equation*}
  \pi^{\alpha}(f^{\alpha})\eta^{\alpha}(x) \geq  \pi^{\alpha}(x_0) \eta^{\alpha}(x) \inf_{a\in A} \{f(x_0,a)\} > 0.
 \end{equation*}
 Further, it follows from~\eqref{eq:prop_main} that for all $\beta\geq\lambda^{\alpha}$ and $t$ sufficiently large,
 \begin{equation*}
  e^{\beta t} \E_x\left[f^{\alpha}\left(\X{\alpha}_t\right)\right] \geq \pi^{\alpha}(f^{\alpha})\eta^{\alpha}(x) - C_0  \e{-\gamma t}.
 \end{equation*}
 The conclusion follows immediatly by integration.
\end{proof}

 We are now in a position to prove Theorem~\ref{thm:main-inf}. Recall that $\lambda^*=\sup_{\alpha\in\Ac_M} {\lambda^{\alpha}}$ is finite since $\lambda^\alpha$ is bounded from above by $\dmax_1$.
 Denote by $\Ac_M^*$ the collection of $\alpha\in\Ac_M$ such that $\lambda^\alpha=\lambda^*$.

 Let us assume for the moment that $\Ac_M^*$ is non-empty. First, in view of Proposition~\ref{prop:main}, it holds for all
 $\beta<\lambda^*$ and $\alpha\in\Ac_M^*$,
 \begin{equation*}
   \left|\left(\lambda^* - \beta\right) J_{\beta}(x,\alpha) - \pi^{\alpha}(f^{\alpha})\eta^{\alpha}(x)\right| \leq C \left(\lambda^* - \beta\right).   
 \end{equation*}  
 It follows that
 \begin{equation*}
   \left(\lambda^* - \beta\right) v_\beta(x) \leq \pi^{\alpha}(f^{\alpha})\eta^{\alpha}(x) + C \left(\lambda^* - \beta\right).   
 \end{equation*}    
 Hence, we obtain
 \begin{equation}\label{eq:upperbound1.5}
   \limsup_{\beta\uparrow\lambda^*} {\left(\lambda^* - \beta\right) v_\beta(x)} \leq \inf_{\alpha\in\Ac_M^*} {\left\{\pi^{\alpha}(f^{\alpha})\eta^\alpha(x)\right\}}.
 \end{equation}
  
 Second, we denote by $\alpha_\beta\in\Ac_M$ an optimal Markov control for the $\beta$--discounted minimization problem as in Theorem~\ref{thm:alpha-opt-Markov-sup}. 
 One clearly has
 \begin{equation*}
  \left(\lambda^* - \beta\right) v_\beta(x) \geq \left(\lambda^{\alpha_\beta} - \beta\right) v_\beta(x). 
 \end{equation*}
   To conclude the proof of Theorem~\ref{thm:main-inf}, it is hence sufficient to show that
   $\Ac^*_M\not=\emptyset$ and that there exists $\alpha^*\in\Ac^*_M$ such that
 \begin{equation}
  \label{eq:main-inf}
  \liminf_{\beta\uparrow\lambda^*} {\left(\lambda^{\alpha_\beta} - \beta\right) v_\beta(x)} = \pi^{\alpha^*}(f^{\alpha^*})\eta^{\alpha^*}(x).
 \end{equation}
 Note that the existence of $\alpha^*$ such that $\lambda^{\alpha^*}=\lambda^*$ will validate the proof
  of~\eqref{eq:upperbound1.5} above.
 
 Let us consider an increasing sequence $(\beta_n)_{n\in\N}$ converging to $\lambda^*$ such that for all $x\in\N$,
 \begin{equation*}
  \lim_{n\to+\infty} {\left(\lambda^{\alpha_{\beta_n}} - \beta_n\right) v_{\beta_n}(x)} = \liminf_{\beta\uparrow\lambda^*} {\left(\lambda^{\alpha_\beta} - \beta\right) v_\beta(x)}.
 \end{equation*}
 Since $A$ is compact, we can extract a subsequence (still denoted $(\beta_n)_{n\in\N}$) such that $\alpha_{\beta_n}$ converges pointwise to
 $\alpha^*\in\Ac_M$. The rest of the proof consists in showing that $\alpha^*\in\Ac_M^*$ and~\eqref{eq:main-inf} is satisfied. 
 To clarify the presentation, we split the proof of this result in three parts.
  
\medskip\noindent\textsl{Step 1.} As a first step, we show that $\alpha^*\in\Ac_M^*$ and the QSD $\pi^{\alpha_{\beta_n}}$ (resp. the extinction rate $\lambda^{\alpha_{\beta_n}}$) converges to $\pi^{\alpha^*}$ (resp. $\lambda^{\alpha^*}$). In view of Lemma~\ref{lem:tight} below, $(\pi^{\alpha_{\beta_n}})_{n\in\N}$ is a tight sequence of probability measures on $\N$. Hence, we can extract a subsequence converging pointwise to some probability measure $\pi^*$ on $\N$. Denote by $\Kc^{\alpha}$ the adjoint operator of $\Lc^{\alpha}$, \textit{i.e.}, for all $u\in\R^{\ZZ_+}$ and $x\in\ZZ_+$,
 \begin{equation*}
  \Kc^{\alpha} {u} (x) = \sum_{y=0}^{x-1} {b^{\alpha}(y) p_{y,x-y}^{\alpha}u(y)}  + d^{\alpha}(x+1) u(x+1) - \left( b^{\alpha}(x) + d^{\alpha}(x) \right) u(x).
 \end{equation*}
 It follows from Proposition~4 of~\cite{meleard-villemonais-12} that
 \begin{equation}\label{eq:pi}
  \lambda^{\alpha_{\beta_n}} \pi^{\alpha_{\beta_n}} + \Kc^{\alpha_{\beta_n}} {\pi^{\alpha_{\beta_n}}} = 0.
 \end{equation}
 In addition, Proposition~\ref{prop:main} ensures that $\lambda^{\alpha_{\beta_n}}>\beta_n$ and thus $\lambda^{\alpha_{\beta_n}}$ converges to $\lambda^*$.
 Hence, by sending $n$ to infinity, we obtain 
 \begin{equation*}
  \lambda^{*} \pi^{*} + \Kc^{\alpha^*} {\pi^{*}} = 0.
 \end{equation*}
 By Proposition~4 of~\cite{meleard-villemonais-12}, we deduce that $\pi^*$ is a QSD of $\X{\alpha^*}$ with extinction rate $\lambda^*$. The conclusion follows by uniqueness of the QSD.

\medskip\noindent \textsl{Step 2.}
 As an intermediate step, we show that $\liminf_{\beta\uparrow\lambda^*} {(\lambda^{\alpha_\beta}-\beta) v_\beta}$ is collinear with $\eta^{\alpha^*}$. Denote $\phi_n := (\lambda^{\alpha_{\beta_n}}-\beta_n)v_{\beta_n}$ and $\phi:=\lim_{n\to+\infty} {\phi_{n}}$. It follows from Theorems~\ref{thm:HJB} and \ref{thm:alpha-opt-Markov-sup} that
 \begin{equation*}
  \beta_n \phi_n + \Lc^{\alpha_{\beta_n}} {\phi_n} + \left(\lambda^{\alpha_{\beta_n}}-\beta_n\right)f^{\alpha_{\beta_n}} = 0.
 \end{equation*}
 By sending $n$ to infinity, we want to derive that
 \begin{equation}\label{eq:hjblim}
  \lambda^* \phi + \Lc^{\alpha^*} {\phi} = 0.
 \end{equation}
 In view of the continuity assumption of Hypothesis~\ref{hyp:main}, it suffices to show that for all $x\in\N$,
 \begin{equation}\label{eq:limserie}
  \lim_{n\to+\infty} {\sum_{k\geq 1} {p_{x,k}^{\alpha_{\beta_n}}\phi_{n}(x+k)}} = \sum_{k\geq 1} {p_{x,k}^{\alpha^*}\phi(x+k)}.
 \end{equation}
 Let us show first that $\phi_{n}$ is bounded uniformly w.r.t. $n$. Applying Propositions~\ref{prop:main} and~\ref{thm:QSD}, we obtain
 \begin{align*}
  \phi_{n}(x+k) 
    & \leq \pi^{\alpha_{\beta_n}}\left(f^{\alpha_{\beta_n}}\right)\eta^{\alpha_{\beta_n}}(x+k) + C\left(\lambda^{\alpha_{\beta_n}} -\beta_n\right) \\
    & \leq \left(C' + 1\right) \lVert f\rVert_{\infty} + C\left(\lambda^{*} -\beta_n\right).
 \end{align*}
 This implies that $\phi_{n}$ is uniformly bounded by $C_0 := (C' + 1) \lVert f\rVert_{\infty} + C\left(\lambda^{*} -\beta_0\right)$.
 In addition, it holds 
 \begin{multline*}
  \left| \sum_{k\geq 1} {p_{x,k}^{\alpha^*}\phi(x+k)} - \sum_{k\geq 1} {p_{x,k}^{\alpha_{\beta_n}}\phi_{n}(x+k)}\right| 
  \\ \leq C_0 \sum_{k\geq 1} {\left| p_{x,k}^{\alpha^*} - p_{x,k}^{\alpha_{\beta_n}}\right|}
    + \sum_{k\geq 1} {p_{x,k}^{\alpha^*}\left|\phi(x+k) - \phi_{n}(x+k)\right|}.
 \end{multline*}
 Using Scheff\'e's lemma and the dominated convergence theorem, we deduce that the relation~\eqref{eq:limserie} is satisfied, and
 hence~\eqref{eq:hjblim}. In other words, $\phi$ is a bounded eigenfunction of $\Lc^{\alpha^*}$ for the eigenvalue $-\lambda^* = -\lambda^{\alpha^*}$. 
 According to Corollary 2.4 of~\cite{champagnat-villemonais-15}, it yields that $\phi$ is collinear with $\eta^{\alpha^*}$.

\medskip\noindent\textsl{Step 3.} To conclude, it remains to identify the collinearity coefficient between $\phi$ and $\eta^{\alpha^*}$. It follows by integration of the HJB equation~\eqref{eq:HJB} that
 \begin{equation*}
  \beta \pi^{\alpha_\beta}\left(v_\beta\right) + \pi^{\alpha_\beta}\left(\Lc^{\alpha_\beta} {v_\beta}\right) + \pi^{\alpha_\beta}\left(f^{\alpha_\beta}\right) = 0.
 \end{equation*}
 In particular, using the relation~\eqref{eq:pi}, we obtain
 \begin{equation*}
  - \pi^{\alpha_{\beta_n}}\left(\phi_n\right) + \pi^{\alpha_{\beta_n}}\left(f^{\alpha_{\beta_n}}\right) = 0.
 \end{equation*}
 Now we can use Scheff\'e's Lemma and the dominated convergence theorem as above to derive
 \begin{equation*}
  - \pi^{\alpha^*}\left(\phi\right) + \pi^{\alpha^*}\left(f^{\alpha^*}\right) = 0.
 \end{equation*}
 Since $\pi^{\alpha^*}(\eta^{\alpha^*})=1$ as stated in Proposition~\ref{thm:QSD}, we conclude that $\phi=\pi^{\alpha^*}\left(f^{\alpha^*}\right)\eta^{\alpha^*}$, and the proof of Theorem~\ref{thm:main-inf} is completed.
 
 \begin{lemma}\label{lem:tight}
  The family of probability measures $(\pi^\alpha)_{\alpha\in\Ac_M}$ is tight.
 \end{lemma}
 
 \begin{proof}
  It follows from Theorem~\ref{thm:QSD-short-version} that for any $K\geq 1$,
  \begin{equation*}
   \sum_{x\geq K} \pi^{\alpha}(x) 
     \leq  C\e{-\gamma t} + \frac{\P_1\left(X^\alpha_t\geq K\right)}{\P_1\left(t <\tau^\alpha\right)}.
  \end{equation*}
  On the one hand, the Markov inequality and~\eqref{eq:moment} yield
  \begin{equation*}
   \P_1\left(X^\alpha_t\geq K\right) \leq \frac{\e{M\bar{b}t}}{K}.
  \end{equation*}
  On the other hand, as already mentioned, $\P_1(t <\tau^\alpha) \geq \e{-\dmax_1 t}$.
  We deduce that, for all $t\geq 0$,
  \begin{equation*}
   \sum_{x\geq K} \pi^{\alpha}(x)  \leq  C\e{-\gamma t} + \frac{\e{(M\bar{b}+\bar{d}_1)t}}{K}.
  \end{equation*}
  This ends the proof.
 \end{proof}

 \subsection{Proof of Theorem~\ref{thm:main-sup}}
 
 The proof of Theorem~\ref{thm:main-sup} follows from similar arguments as those developed in Section~\ref{sec:proofthmmain}. In
 particular, it relies on the HJB equation and the existence of an optimal Markov control stated in Theorem~\ref{thm:alpha-opt-Markov}. To apply
 this result, we need to show that $w_\beta$ is bounded for $\beta<\lambda_*=\inf_{\alpha\in\Ac_M} {\lambda^\alpha}$. This is a delicate issue and we postpone its proof to
 Lemma~\ref{lem:bounded} below.
  
  Similar to the previous section, we start by assuming that there exists at least one control $\alpha\in\Ac_M$ such that $\lambda^\alpha=\lambda_*$. Using Proposition~\ref{prop:main}, one easily checks that
  \begin{equation*}
   \liminf_{\beta\uparrow\lambda_*} {\left(\lambda_* - \beta\right) w_\beta(x)} \geq \sup_{\substack{\alpha\in\Ac_M \\ \lambda^\alpha=\lambda_*}} {\left\{\pi^{\alpha}(f^{\alpha})\eta^\alpha(x)\right\}}.
  \end{equation*}
  
  To conclude the proof, it is hence enough to show that
  there exists $\alpha_*\in\Ac_M$ such that $\lambda^{\alpha_*}=\lambda_*$ and
  \begin{equation}
   \label{eq:main-sup}
   \limsup_{\beta\uparrow\lambda_*} {\left(\lambda_* - \beta\right) w_\beta(x)} = \pi^{\alpha_*}(f^{\alpha_*})\eta^{\alpha_*}(x).
  \end{equation}
  For the sake of clarity, we split the proof in three steps. 
  
  \medskip\noindent\textsl{First Step.} Let us show first that there exists $\alpha\in\Ac_M$ such that $\lambda^{\alpha}=\lambda_*$. To achieve this, we want to extend the arguments of Step 1 in the proof of Theorem~\ref{thm:main-inf}. Denote by $\alpha_\beta$ an optimal Markov control for the $\beta$--discounted maximization problem as in Theorem~\ref{thm:alpha-opt-Markov}. In order to repeat Step 1, the only issue is to prove that --- under appropriate assumptions on $f$ --- $\lambda^{\alpha_\beta}$ converges to $\lambda_*$ as $\beta$ goes to $\lambda_*$. Notice first that it follows from Propositions~\ref{prop:main} and \ref{thm:QSD} that for all $\beta<\lambda_*$,
  \begin{equation}
   \label{eq: last-proof}
   (\lambda^{\alpha_\beta} - \beta) w_\beta(x) \leq (C'+1)\|f\|_{\infty} + C(\lambda^*-\beta).
  \end{equation}
  Hence, we can conclude by showing that, for a specific function $f$, $w_\beta$ tends to $+\infty$ as $\beta$ goes to $\lambda_*$. Here we take $f(x,a)=1$ for all $x\in\N$ and $a\in A$. In view of Proposition~\ref{prop:main}, we have for all $\beta<\lambda_*$,
  \begin{equation*}
   J_\beta\left(x,\alpha\right) 
      \geq \frac{1}{\lambda^{\alpha}-\beta}  \inf_{\alpha\in\Ac_M} {\left\{\eta^{\alpha}(x)\right\}} - C,
  \end{equation*}
  where $\inf_{\alpha\in\Ac_M} {\left\{\eta^{\alpha}(x)\right\}}>0$ by Proposition~\ref{prop:minoration}.
  By considering a family of $\alpha\in\Ac_M$ such that $\lambda^{\alpha}$ converges to $\lambda_*$, we deduce that 
  \begin{equation*}
   w_\beta\left(x\right) \geq \frac{1}{\lambda_*-\beta}  \inf_{\alpha\in\Ac_M} {\left\{\eta^{\alpha}(x)\right\}} - C,
  \end{equation*}
  Hence, the conclusion follows immediately by sending $\beta$ to $\lambda_*$.
  
  \medskip\noindent\textsl{Second Step.} Let us show next that~\eqref{eq:main-sup} holds true under the following assumption:
  \begin{equation*}
   \sup_{\substack{\alpha\in\Ac_M \\ \lambda^\alpha = \lambda_*}}{\left\{\pi^{\alpha}(f^{\alpha})\right\}} > 0.
  \end{equation*}
  Note that this condition is satisfied under Hypothesis~\ref{hyp:main-inf}. The idea is to extend the arguments of the previous section to prove that 
  \begin{equation*}
   \limsup_{\beta\uparrow\lambda_*} {\left(\lambda^{\alpha_\beta} - \beta\right) w_\beta(x)} = \pi^{\alpha_*}(f^{\alpha_*})\eta^{\alpha_*}(x),
  \end{equation*}
  where $\alpha_{\beta}$ is an optimal Markov control for the $\beta$--discounted maximization problem as in Theorem~\ref{thm:alpha-opt-Markov}. In order to repeat Step 1 of Section~\ref{sec:proofthmmain}, the only issue is to show that $\lambda^{\alpha_{\beta}}$ converges to $\lambda_*$ as $\beta$ goes to $\lambda_*$. To achieve this, we can extend the arguments above. Indeed, it follows from Proposition~\ref{prop:main} that
  \begin{equation*}
   w_\beta\left(x\right) \geq \frac{1}{\lambda_* -\beta} \sup_{\substack{\alpha\in\Ac_M \\ \lambda^\alpha = \lambda_*}}{\left\{\pi^{\alpha}(f^{\alpha})\eta^{\alpha}(x)\right\}} - C.
  \end{equation*}
  We deduce that $w_\beta$ tends to $+\infty$ as $\beta$ goes to $\lambda_*$. Together with~\eqref{eq: last-proof}, this imposes that $\lambda^{\alpha_\beta}$ tends to $\lambda_*$. In order to repeat Step 2 of Section~\ref{sec:proofthmmain}, we have to ensure that 
  \begin{equation*}
   \limsup_{\beta\uparrow\lambda_*} {\left(\lambda^{\alpha_\beta} - \beta\right) w_\beta(x)} < +\infty.
  \end{equation*}
  This is a straightforward consequence of~\eqref{eq: last-proof}. The rest of the proof follows easily by repeating the arguments of Section 6.1.
  
  \medskip\noindent\textsl{Third Step.} Let us show finally that~\eqref{eq:main-sup} holds true under the following assumption:
  \begin{equation*}
   \sup_{\substack{\alpha\in\Ac_M \\ \lambda^\alpha = \lambda_*}}{\left\{\pi^{\alpha}(f^{\alpha})\right\}} = 0.
  \end{equation*}
  We argue by considering two cases. If $\lambda^{\alpha_\beta}$ converge to $\lambda_*$, we can once again repeat the arguments of the previous section to reach the conclusion. If $\lambda^{\alpha_\beta}$ does not converges to $\lambda_*$, then it follows from~\eqref{eq: last-proof} and the monotonicity of $\beta\to w_\beta(x)$  that for all $\beta<\lambda_*$,
  \begin{equation*}  
   w_{\beta}(x)\leq \lim_{\beta\uparrow\lambda_*} {w_\beta} (x) < +\infty.
  \end{equation*} 
  Hence, \eqref{eq:main-sup} holds true once again.

\begin{lemma}\label{lem:bounded}
 For every $\beta<\lambda_*$, the value function $w_\beta$ is bounded.
\end{lemma}

\begin{proof}
  If we knew that the optimal control in the value function is Markov, this would directly follow from
  Proposition~\ref{prop:main}. The difficulty is to prove that non-Markov controls cannot give unbounded costs.
Denote
\begin{equation*}
 \beta_* := \inf \left\{ \beta;\ \sup_{x\in\N} w_\beta(x) = +\infty\right\}.
\end{equation*}
 We assume that $\beta_*<\lambda_*$ in order to reach a contradiction. 
 
 \medskip\noindent \textsl{First step.} We start by showing that $w_{\beta_*}$ is bounded. 
 For all $\beta<\beta_*$, it follows from Theorem~\ref{thm:alpha-opt-Markov} that the optimal control is Markov, and hence from
 Propositions~\ref{thm:QSD} and \ref{prop:main} that
 \begin{equation*}
  w_\beta(x) = J_\beta(x,\alpha_\beta) 
    \leq \frac{\pi^{\alpha_\beta}(f^{\alpha_\beta})}{\lambda^{\alpha_\beta}-\beta}\eta^{\alpha_\beta}(x) + C
    \leq \frac{\left(C'+1\right)\|f\|_{\infty}}{\lambda_*-\beta} + C.
 \end{equation*}
 We deduce that
 \begin{equation*}
  w_{\beta_*}(x) = \lim_{\beta\uparrow \beta_*} w_\beta(x) \leq \frac{\left(C'+1\right)\|f\|_{\infty}}{\lambda_*-\beta_*} + C < +\infty.
 \end{equation*}
 Note that the first identity above relies on a permutation of supremum since $\beta\mapsto w_{\beta}(x)$ is nondecreasing for all $x\in\N$.
 
 \medskip\noindent \textsl{Second step.} To reach a contradiction, it remains to prove that there exists $\varepsilon>0$ such that 
 \begin{equation*}
  \sup_{x\in\N} w_{\beta_*+\varepsilon}(x) < +\infty.
 \end{equation*}
 Without loss of generality we take $f=\mathds{1}_{\N}$ and we claim that for all $\alpha\in\Ac$ and $x\in\N$,
 \begin{equation}\label{eq:induction}
  \int_0^{+\infty} {s^n e^{\beta_* s} \P_x\left(s<\tau^{\alpha}\right)  ds} \leq   C_0^{n+1} n!,
 \end{equation}
 where 
 \begin{equation*}
  C_0:= \sup_{x\in\N} \sup_{\alpha\in\Ac} \left\{\int_0^{+\infty} {e^{\beta_* s} \P_x\left(s<\tau^{\alpha}\right)  ds}\right\}.
 \end{equation*}
 If~\eqref{eq:induction} is proved, the conclusion follows from the fact that, for any $\varepsilon<\frac{1}{C_0}$,
 \begin{equation*}
  w_{\beta_*+\varepsilon}(x) 
     = \sup_{\alpha\in\Ac} \left\{\int_0^{+\infty} {e^{(\beta_*+\varepsilon)s} \P_x\left(s<\tau^{\alpha}\right)  ds}\right\}
     \leq C_0 \sum_{n\geq 0} (C_0\varepsilon)^n = \frac{C_0}{1-\varepsilon C_0}.
 \end{equation*}

 Let us show that the relation~\eqref{eq:induction} holds by induction. The result is obvious for $n=0$. Using successively Fubini-Tonelli and the pseudo-Markov property (Lemma~\ref{lem:cond}), the incrementation step follows by
 \begin{multline*}
  \int_0^{+\infty} {s^{n+1} e^{\beta_* s} \P_x\left(s<\tau^{\alpha}\right)  ds} \\
  \begin{aligned}
    & = \int_0^{+\infty} {\int_t^{+\infty} {s^{n} e^{\beta_* s} \P_x\left(s<\tau^{\alpha}\right)  ds}\, dt} \\
    & = \int_0^{+\infty} { \int_0^{+\infty} {(t+s)^{n} e^{\beta_* (t+s)} \E_{x}\left[\P_{X^{\alpha}_t}\left(s<\tau^{\alpha^t}\right)\right]  ds}\, dt} \\
    & = \sum_{k=0}^{n} \binom{n}{k} \int_0^{+\infty} {t^k e^{\beta_* t} \int_0^{+\infty} {s^{n-k} e^{\beta_* s} \E_{x}\left[\P_{X^{\alpha}_t}\left(s<\tau^{\alpha^t}\right)\right]  ds}\, dt} \\ 
    & = \sum_{k=0}^{n} \binom{n}{k} \int_0^{+\infty} { t^k e^{\beta_* t} \sum_{y\in\N} {\E_{x}\left[ \mathds{1}_{X^{\alpha}_t=y} \int_0^{+\infty} {s^{n-k} e^{\beta_* s} \P_{y}\left(s<\tau^{\alpha^{t}}\right) ds}\right]   dt} }
  \end{aligned}
 \end{multline*}
 It follows that 
 \begin{gather*}
  C_{n+1} \leq  \sum_{k=0}^{n} \binom{n}{k} C_{k} C_{n-k},
  \intertext{where}
  C_k := \sup_{y\in\N} \sup_{\alpha\in\Ac} \int_0^{+\infty} { s^k e^{\beta_* s} \P_y\left(s<\tau^{\alpha}\right)  ds}.
 \end{gather*}
 To conclude, it remains to observe that the sequence $u_n := C_0^{n+1} n!$ is the solution of 
 \begin{equation*}
  u_{n+1} =  \sum_{k=0}^{n} \binom{n}{k} u_{k} u_{n-k}
 \end{equation*}
 satisfying $u_0 = C_0$.  
\end{proof}

\paragraph{Aknowledgement} Julien Claisse gratefully acknowledges financial support from the ERC Advanced Grant 321111 ROFIRM.

\bibliography{biblio-math}{}
\bibliographystyle{alpha}

\end{document}